\documentclass[12pt,a4paper]{scrartcl}
\usepackage[margin=25.4mm]{geometry}
\usepackage{latexsym} 
\usepackage{amsmath}
\usepackage{amsthm}
\usepackage{amsfonts}
\usepackage{amssymb}
\usepackage{amsxtra}
\usepackage{amscd}
\usepackage[shortlabels]{enumitem}
\usepackage{mathrsfs}
\usepackage{hyperref}
\usepackage{mathtools}
\usepackage{ifthen}
\usepackage{xcolor}

\hypersetup{
pdfauthor={},
pdftitle={},
breaklinks=true,
colorlinks=true,
linkcolor=blue,
citecolor=blue,
urlcolor=blue,
filecolor=blue,
}

\numberwithin{equation}{section} 

\newenvironment{pdeq}{ \left\{ \begin{aligned}}{\end{aligned}\right.}
\newcommand{\np}[1]{(#1)}
\newcommand{\nb}[1]{[#1]}
\newcommand{\bp}[1]{\big(#1\big)}
\newcommand{\bb}[1]{\big[#1\big]}
\newcommand{\Bp}[1]{\bigg(#1\bigg)}

\newcommand{\calp}{{\mathcal P}}

\newcommand{\R}{\mathbb{R}}

\newcommand{\N}{\mathbb{N}}

\DeclareMathOperator{\e}{e}

\DeclareMathOperator{\Div}{div}

\DeclareMathOperator{\supp}{supp}

\DeclareMathOperator{\trace}{Tr}

\newcommand{\embeds}{\hookrightarrow}

\newcommand{\set}[1]{\ensuremath{\{#1\}}}

\newcommand{\setcl}[2]{\ensuremath{\bigl\{#1\ \big\vert\ #2\bigr\}}}
\newcommand{\setcL}[2]{\ensuremath{\biggl\{#1\ \bigg\vert\ #2\biggr\}}}

\newcommand{\closure}[2]{\overline{#1}^{#2}}

\renewcommand{\restriction}[2]{#1 | _{#2}}
\newcommand{\transpose}{\top}

\newcommand{\idmatrix}{I}
\newcommand{\OmegaT}{\Omega\times(0,T)}
\newcommand{\OmegaTprime}{\Omega\times(0,T')}
\newcommand{\OmegaTk}{\Omega\times(0,T_k)}
\newcommand{\Omegat}{\Omega\times(0,t)}
\newcommand{\OmegaTzero}{\Omega\times[0,T)}
\newcommand{\OmegaTzerocl}{\overline{\Omega}\times[0,T)}
\newcommand{\partialOmegaT}{\partial\Omega\times(0,T)}

\newcommand{\ddt}{\frac{{\mathrm d}}{{\mathrm d}t}}

\newcommand{\grad}{\nabla}

\newcommand{\pt}{\partial_t}

\newcommand{\dx}{{\mathrm d}x}

\newcommand{\dtau}{{\mathrm d}\tau}

\newcommand{\dt}{{\mathrm d}t}

\newcommand{\dsigma}{{\mathrm d}\sigma}

\newcommand{\convder}{\mathrm{D}_t}
\newcommand{\jaumannder}[1]{\overset{\text{\tiny\raisebox{-0.2em}{$\triangledown$}}}{#1}}
\newcommand{\norm}[1]{\lVert#1\rVert}
\newcommand{\norml}[1]{\big\lVert#1\big\rVert}

\newcommand{\snorm}[1]{{\lvert #1 \rvert}}

\newcommand{\snormL}[1]{{\Bigl\lvert #1 \Big\rvert}}
\newcommand{\LR}[1]{\mathrm{L}^{#1}}

\newcommand{\LRloc}[1]{\mathrm{L}^{#1}_{\mathrm{loc}}} 
\newcommand{\CR}[1]{\mathrm{C}^{#1}}  
\newcommand{\CRi}{\CR \infty}
\newcommand{\CRci}{\CR \infty_0}
\newcommand{\CRc}[1]{\CR{#1}_0}
\newcommand{\CRweak}{\CR{}_{\mathrm{w}}}

\newcommand{\HSR}[1]{\mathrm{H}^{#1}} 
\newcommand{\HSRN}[1]{\mathrm{H}^{#1}_0}

\newcommand{\LRsigma}[1]{\mathrm{L}^{#1}_{\sigma}}

\newcommand{\HSRNsigma}[1]{\mathrm{H}^{#1}_{0,\sigma}}

\newcommand{\CRcisigma}{\CR{\infty}_{0,\sigma}}

\newcommand{\LH}{\mathrm{LH}_T}
\newcommand{\LHprime}{\mathrm{LH}_{T'}}
\newcommand{\LRdev}[1]{\mathrm{L}^{#1}_{\delta}} 
\newcommand{\HSRdev}[1]{\mathrm{H}^{#1}_{\delta}} 
\newcommand{\CRcidev}{\CR{\infty}_{0,\delta}}
\newcommand{\CRidev}{\CR{\infty}_{\delta}}
\newcommand{\XT}{\mathrm{X}_T}
\newcommand{\XTprime}{\mathrm{X}_{T'}}
\newcommand{\TestFn}{Z_{T'}}
\newcommand{\nsnonlinb}[2]{#1\cdot\grad #2}
\newcommand{\nsnonlin}[1]{\nsnonlinb{#1}{#1}}
\newcommand{\vvel}{v}
\newcommand{\vpres}{p}

\newcommand{\Vvel}{V}
\newcommand{\Vpres}{P}

\newcommand{\wvel}{w}
\newcommand{\wpres}{\mathfrak{q}}

\newcommand{\uvel}{u}

\newcommand{\straintensor}{D}
\newcommand{\rottensor}{W}
\newcommand{\Stens}{S}
\newcommand{\Ttens}{T} 
\newcommand{\tin}{\text{in }}

\newcommand{\ton}{\text{on }}

\newcommand{\eqrefsub}[2]{\ensuremath{\eqref{#1}_{#2}}}
\newcommand{\half}{\frac{1}{2}}

\newcommand{\nvec}{\mathrm{n}}

\newcommand{\potential}{\calp}
\newcommand{\newCCtr}[2][d]{
\newcounter{#2}\setcounter{#2}{0}
\expandafter\xdef\csname kyedtheconst#2\endcsname{#1}
}
\newcommand{\Cc}[2][nolabel]{
\stepcounter{#2}
\expandafter\ensuremath{\csname kyedtheconst#2\endcsname_{\arabic{#2}}}
\ifthenelse{\equal{#1}{nolabel}}
{}
{\expandafter\xdef\csname kyedconst#1\endcsname
{\expandafter\ensuremath{\csname kyedtheconst#2\endcsname_{\arabic{#2}}}}}
}
\newcommand{\Ccn}[2][nolabel]{
\expandafter\ensuremath{\csname kyedtheconst#2\endcsname}
\ifthenelse{\equal{#1}{nolabel}}
{}
{\expandafter\xdef\csname kyedconst#1\endcsname
{\expandafter\ensuremath{\csname kyedtheconst#2\endcsname}}}
}
\newcommand{\CcSetCtr}[2]{
\setcounter{#1}{#2}
}
\newcommand{\Cclast}[1]{
\expandafter\ensuremath{\csname kyedtheconst#1\endcsname_{\arabic{#1}}}
}
\newcommand{\Ccllast}[1]{
\addtocounter{#1}{-1}
\expandafter\ensuremath{\csname kyedtheconst#1\endcsname_{\arabic{#1}}}
\addtocounter{#1}{1}
}
\newcommand{\const}[1]{
\expandafter{\ifcsname kyedconst#1\endcsname
  \csname kyedconst#1\endcsname
\else
  \errmessage{Undefined Kyedconstant #1.}%
\fi}
}

\theoremstyle{plain}
\newtheorem{thm}{Theorem}[section]
\newtheorem{lem}[thm]{Lemma}
\newtheorem{prop}[thm]{Proposition}
\newtheorem{cor}[thm]{Corollary}
\theoremstyle{definition}
\newtheorem{defn}[thm]{Definition}
\newtheorem{asm}[thm]{Assumption}
\theoremstyle{remark}
\newtheorem{rem}[thm]{Remark}

\newcommand{\hs}{{\LRdev{2}(\Omega)}}

\newcommand{\dd}{\mathrm{d}} 
 
\newcommand{\Om}{\Omega}
\newcommand{\ve}{\varepsilon}

\newcommand{\tim}{{\times}}
\renewcommand{\dx}{\,{\mathrm d}x}
\renewcommand{\dtau}{\,{\mathrm d}\tau}
\renewcommand{\dt}{\,{\mathrm d}t}
\newcommand{\MR}[1]{}  
\newcommand{\bysame}{\rule{3em}{1pt}}  

\begin{document}
\newCCtr[C]{C}
\newCCtr[M]{M}
\newCCtr[c]{c}
\let\oldproof\proof
\def\proof{\CcSetCtr{c}{-1}\oldproof} 

\title{Leray--Hopf solutions to \\
a viscoelastoplastic fluid model\\ 
with nonsmooth stress-strain relation\thanks{The research was partially supported by Deutsche Forschungsgemeinschaft (DFG) through SFB 1114 {\em Scaling Cascades in Complex Systems}  (Project No.\ 235221301) via the subproject B01 
``Fault networks and scaling properties of deformation accumulation''}}

\author{
Thomas Eiter%
\thanks{Weierstrass Institute for Applied Analysis and Stochastics, Mohrenstra\ss{}e 39, 10117 Berlin, Germany}
\and 
Katharina Hopf%
\footnotemark[2]
\and 
Alexander Mielke%
\footnotemark[2]
\thanks{Institut f\"ur Mathematik, Humboldt-Universit\"at zu Berlin, Rudower Chaussee 25, 12489 Berlin, Germany}
}

\date{\today}
\maketitle

\begin{abstract} 
We consider a fluid model including viscoelastic and viscoplastic effects. The
state is given by the fluid velocity and an internal stress tensor that is
transported along the flow with the Zaremba--Jaumann derivative. Moreover, the
stress tensor obeys a nonlinear and nonsmooth dissipation law as well as stress
diffusion. We prove the existence of global-in-time weak solutions satisfying an energy inequality under general Dirichlet
conditions for the velocity field and Neumann conditions for the stress tensor. 
\end{abstract}



\section{Introduction}
\label{sec:Introduction}

In this article we investigate the equations of motion
that describe the flow of a viscoelastoplastic fluid with stress diffusion
modeled in the following way.
On a time interval  $(0,T)$ and a bounded domain $\Omega\subset\R^3$, 
we consider the system of equations
\begin{equation}\label{sys:PDE}
\begin{pdeq}
\rho \convder{\Vvel}-\Div\bp{\eta_1\Stens+2\mu\straintensor(\Vvel)-\Vpres\idmatrix}
&=F
&&\tin \OmegaT, \\[0.1em]
\Div\Vvel
&=0
&&\tin \OmegaT, \\[-0.3em]
\jaumannder{\Stens}+\partial\potential(\Stens)
-\gamma\Delta\Stens
&\ni\eta_2 \straintensor(\Vvel)
&&\tin\OmegaT.
\end{pdeq}
\end{equation}
Here the first two equations describe the flow of an incompressible fluid with
Eulerian velocity field $\Vvel\colon\OmegaT\to\R^3$ and pressure field
$\Vpres\colon\OmegaT\to\R$ affected by a prescribed external force
$F\colon\OmegaT\to\R^3$.  The relevant Cauchy stress tensor
$\mathbb T=\eta_1 S+ 2\mu\straintensor(\Vvel)-\Vpres\idmatrix$ consists of the
classical term $2\mu\straintensor(\Vvel)-\Vpres\idmatrix$ for Newtonian fluids
and an extra stress tensor
\[
  \Stens\colon\OmegaT\to\R^{3\times3}_{\delta}
\coloneqq\setcl{M\in\R^{3\times3}}{M=M^\transpose,\ \trace M=0},
\]
which satisfies the additional evolution equation \eqrefsub{sys:PDE}{3} and is
thus subject to a special transport encoded in $\jaumannder{\Stens}$ along the
velocity field $\Vvel$, a nonlinear dissipation law via
$\partial\potential(\Stens)$, and another diffusion process.  Here $\rho$,
$\eta_1$, $\eta_2$, $\mu$ and $\gamma$ denote positive constants, and
$\straintensor(\Vvel)\coloneqq \half\np{\grad\Vvel+\grad\Vvel^\transpose}$
denotes the symmetric rate-of-strain tensor. 
Following \cite{MoDuMu02MCMV,GerYue07RCMM,HerrenGeryaVand2017IRSDF,PHGAV2019SAFG}, we choose $\Stens \in
\R^{3\times3}_{\delta} $ to be the deviatoric stress tensor (i.e.\  $\trace
\Stens=0$), which corresponds to the incompressibility of the fluid encoded in
\eqref{sys:PDE}$_2$, such that the pressure $P$ contains the full spherical
part of the Cauchy stress tensor. For other modeling choices we refer to
Section \ref{su:Modeling}. 

System \eqref{sys:PDE} is complemented by boundary and initial conditions. The
former are given by 
\begin{equation}
  \label{eq:BoundaryConditions}
  \Vvel\cdot\nvec=0, \qquad \Vvel-(\Vvel\cdot\nvec)\nvec=g, 
  \qquad \nvec\cdot\grad\Stens=0 \qquad \ton \partialOmegaT,
\end{equation}
which means that there is no boundary flux, that the tangential part of the
fluid velocity at the boundary coincides with some prescribed function
$g\colon\partialOmegaT\to\R^3$ and that $\Stens$ has vanishing normal
derivative.  The first two conditions can also be summarized as $\Vvel=g$ on
$\partialOmegaT$ for some $g$ with $g\cdot\nvec=0$.  Note that, from a physical
point of view, only the case $g\cdot\nvec=0$ seems to fit to the Neumann
boundary condition $\nvec\cdot\grad\Stens=0$.  If we would allow for
$g\cdot n\neq 0$, more complicated boundary conditions for $S$ would be needed.
The initial conditions are
\begin{equation}
  \label{eq:InitialConditions}
\Vvel(\cdot,0)=\Vvel_0,\qquad\Stens(\cdot,0)=\Stens_0 \qquad\tin \Omega.
\end{equation}

For a fluid velocity $\Vvel$, the material derivative is given by
\[
\convder A\coloneqq \pt A+ \nsnonlinb{\Vvel}{A},
\]
and as an objective derivative of a tensor $\Stens$ we use the Zaremba--Jaumann derivative
\[
\jaumannder \Stens
\coloneqq \convder\Stens+\Stens\,\rottensor(\Vvel)-\rottensor(\Vvel)\,\Stens
=\pt\Stens+\nsnonlinb{\Vvel}{\Stens}+\Stens\,\rottensor(\Vvel)
 -\rottensor(\Vvel)\,\Stens,
\]
also called co-rotational derivative, where
$\rottensor(\Vvel)\coloneqq\half\np{\grad\Vvel-\grad\Vvel^\transpose}$.  Note
that this choice of the objective derivative is not canonical and there are
different ways to define objective derivatives for tensors.  However, the
choice made here is commonly used in geodynamics (cf.\
\cite{MoDuMu02MCMV,GerYue07RCMM,HerrenGeryaVand2017IRSDF,PHGAV2019SAFG})
and comes along with special features that are very useful for the
mathematical analysis, see below.
 
The mathematical study of viscoelastic fluids with different choices of the
objective derivatives (including the upper and lower convected Maxwell
derivatives) started in the middle 1980s, see e.g.\ \cite{JoReSa85HCTF,
  RenRen86LSPC, ReHrNo87MPV, CooSch91ILFV, Rena00MAVF}. Because of the strong
nonlinearities arising from the objective derivatives, a first global existence
result was only established years later in \cite{LioMas00GSOM} based on the
Zaremba--Jaumann derivative and a linear dissipation law
$\partial\potential(\Stens)= a \Stens$ with $a>0$. More recently, the more
difficult case of a Maxwell fluid with $\mu=0$ (and without stress diffusion,
i.e.\ $\gamma =0$) has also been considered, see~\cite{CLLM_2019} and
references therein.

For more general nonlinear situations there is a series of works involving
implicitly defined stress-strain relations of the type
$G(\Stens, \straintensor(\Vvel))=0$, see
\cite{BulicekGwiazdaMalekSwierczewskaGwiazda_UnsteadyFlowsImplicitlyConsitutedIncompFluids_2012}
and the references in the recent survey \cite{BlMaRa2020CIFM}. Viscoelastic
fluids have a constitutive relation of rate-type, i.e.\ they involve suitable
convective derivatives of the strain tensor $\straintensor(\Vvel)$ or of the
stress tensor, as in our equation \eqref{sys:PDE}$_{3}$. The treatment of such
nonlinearities is possible by using the recently introduced regularization of
stress diffusion, i.e.\ $\gamma>0$, as first illustrated in~\cite{BMPS2018PDEA}
for a simplified model replacing the tensor evolution by a scalar problem.  We
refer to \cite{MPSS18TVRT} for a careful thermodynamical modeling of such
viscoelastic fluids and to \cite{BaBuMa21LDET}, where a large data global
existence result for weak solutions was obtained for a one-parameter family of
convected tensor derivatives including the (simpler) case of the
Zaremba--Jaumann rate.

Our work is in a similar spirit as the latter one, but it generalizes the 
conventional linear or quadratic stress-strain relation
by allowing in~\eqref{sys:PDE}$_{3}$ for subdifferentials
\[
\begin{aligned}
\Stens \mapsto \partial \potential(\Stens) = \setcL{ A \in
	\LR{2}(\Om;\R^{3\times3}_{\delta})
}{ &\potential(\widetilde S) \geq \potential(S)
  + \int_\Omega A\colon
  (\widetilde S{-}S) \dx \\
 &\qquad\qquad\quad  \text{ for all }\widetilde S \in
 \LR{2}(\Om;\R^{3\times3}_{\delta})
} 
\end{aligned}
\]
of a general dissipation potential
$\potential: \LR{2}(\Om;\R^{3\times3}_{\delta}) \to [0,\infty]$ that is convex,
lower semicontinuous and satisfies $\potential(0)=0$.  The space
$\LR{2}(\Om;\R^{3\times3}_{\delta})$ is a natural choice in virtue of the
formal energy-dissipation balance~\eqref{eq:401} below. While
smooth stress-strain relations of polynomial type seem to be sufficient for the
modeling of polymeric fluids (cf.\
\cite{Rena00MAVF,BMPS2018PDEA,MPSS18TVRT,BaBuMa21LDET}), such nonsmooth
dissipation potentials are important for viscoelastoplastic fluid models that
are used in geodynamics for the deformation of rocks in lithospheric plates,
namely
\begin{equation}
  \label{eq:pot.example.intro}
\potential (S)=\int_\Omega \mathfrak P(S(x))\dx  \quad \text{with } 
 \mathfrak P(S) = \left\{\begin{array}{cl} \frac a2 |S|^2 & \text{for } 
    |S|\leq \sigma_\text{yield}, \\ 
\infty & \text{for } |S|> \sigma_\text{yield},
\end{array}  \right.
\end{equation}
where the yield stress $\sigma_\text{yield}>0$ determines the onset of
plastic flow behavior, see \cite{MoDuMu02MCMV,GerYue07RCMM} and
Section \ref{su:Modeling}. Observe that the potential defined
in~\eqref{eq:pot.example.intro} is indeed convex, lower semicontinuous in
$\LR{2}(\Om;\R^{3\times3}_{\delta})$ and satisfies $\potential(0)=0$.

In the context of geodynamics, it is also crucial to
allow for nontrivial boundary data $g\neq 0$ in \eqref{eq:BoundaryConditions},
because often the prescribed drifts of tectonic plates act as boundary data for the
specific region of interest.    

The basic features of the model include, of course, all the difficulties of the
three-dimensional Navier--Stokes equations such that we cannot expect better
solutions than Leray--Hopf solutions for the velocity component $\Vvel$. 
For $\gamma>0$, the equation \eqref{sys:PDE}$_3$ for the stress tensor $\Stens$
is a (semilinear) parabolic equation with linear source term
$\straintensor(\Vvel)$, but, crucially, also coupled nonlinearly to $\Vvel $ via the
Zaremba--Jaumann derivative $\jaumannder{\Stens}$. 

In our analysis we essentially exploit the fact that for sufficiently smooth
functions $\Vvel$ and $\Stens$ satisfying $\nvec \cdot V=0$ on
$\partial \Omega$ we have the identity
\begin{equation}
	\label{eq:Jaum.Identity}
	\frac{\mathrm d}{\mathrm d t} \int_\Omega \frac12|S(t,x)|^2 \dx = 
	\int_\Omega \partial_t S : S \dx = \int_\Omega \convder S : S \dx =
	\int_\Omega \jaumannder{\Stens} : \Stens \dx.
\end{equation}
Exploiting this identity and assuming for the moment that $\Vvel=0$ on $\partial\Omega$, 
one can show that smooth solutions satisfy the energy-dissipation
balance
\begin{align}
	\nonumber
	&\int_\Omega\Big(\frac\rho2 |V(t)|^2 + \frac{\eta_1}{2\eta_2}|\Stens(t)|^2 \Big) \dx  
	\\
	\label{eq:401}
	&\qquad+ \int_0^t\!\! \int_\Omega\!\big(2\mu|\straintensor(\Vvel)|^2 + \frac{\eta_1}{\eta_2} \Stens
	: \partial \potential(\Stens) + \frac{\eta_1\gamma}{\eta_2} \big| \nabla
	\Stens \big|^2 \Big)  \dx \dtau \\
	\nonumber
	& \qquad\qquad  = \int_\Omega\Big(\frac\rho2 |V_0|^2 +
	\frac{\eta_1}{2\eta_2}|\Stens_0|^2 \Big) \dx  + \int_0^t \int_\Omega \Vvel \cdot F \dx \dtau
\end{align} 
for all $t\in(0,T)$.
For the more general case with nontrivial boundary data, we refer to 
\eqref{eq:GenSol.EDI}.
We clearly see how the quadratic energy consisting of the kinetic energy and an
elastic energy associated with $\Stens$ can be changed by the external force
$F$ and is dissipated by three mechanisms: (i) a direct fluid viscosity given
by $\mu>0$, (ii) a stress dissipation encoded in the dissipation potential
$\potential $, and (iii) the stress diffusion associated with $\gamma>0$. 

To simplify the notation we will fix 
two constants
and choose $\rho=1$ and $\eta_1=\eta_2=\eta$ subsequently. With this choice the
quadratic energy is simply given as one half of the $\LR2$ norm of
$(\Vvel,\Stens)$. 
We should note that this choice of parameters rules out the (non-trivial) case $\eta_1\not=0$ and $\eta_2=0$, in which the system is still fully coupled.
In this case, we need to replace the weight $\frac{\eta_1}{2\eta_2}$ in the energy density appearing in~\eqref{eq:401} by some positive constant, leading to an extra term of the form 
$-\eta_1\int_0^t \int_\Omega S:\nabla\Vvel \dx \dtau$ on the right-hand side of~\eqref{eq:401}. However, as will become clear from the proofs, our existence result can  be extended to this case.

\paragraph{Brief description of the main results}
In this article, we perform a large-data global-in-time existence analysis for
system~\eqref{sys:PDE}--\eqref{eq:InitialConditions} that is valid for the
general class of convex potentials $\potential$ specified above. See
Theorem~\ref{thm:genP} for the main result.  A key novelty with respect to
existing literature is the ability to deal with $\potential$ nonsmooth, in
which case the differential inclusion~\eqrefsub{sys:PDE}{3} cannot be replaced
by an equation. Determining an appropriate notion of solution that allows for a
reasonable existence theory is part of our results (see
Sec.~\ref{ssec:gen.sc}).  The inhomogeneous time-dependent boundary data for
the velocity field introduce additional technicalities.  Here, we adapt a
construction for the incompressible Navier--Stokes equations that allows us to
deduce a global-in-time energy control provided that, in some averaged sense,
the data decay to zero as $t\to\infty$, see
Corollary~\ref{cor:GlobalSolutionsEnergy}.  The proof of the existence result
for nonsmooth potentials relies on approximation by a suitable family of
$C^{1,1}$-smooth potentials with the property that
$\Stens\mapsto \partial \potential(\Stens)$ is monotone and globally Lipschitz
continuous. As long as the potential $\potential$ is smooth, somewhat better
results can be obtained, see Theorem~\ref{thm:existence_weaksol}.

\paragraph{Outline}

In Section~\ref{sec:Preliminaries} we specify the notations used in the
subsequent analysis.  The statements of our main results including core
definitions and hypotheses can be found in Section~\ref{sec:main.results},
which also provides some details on the strategy of the proofs and some
further remarks on the modeling. The proof of
Theorem~\ref{thm:existence_weaksol} concerning smooth potentials is provided in
Section~\ref{sec:ex.smooth}.  Section~\ref{su:GenSolNonsmooth} contains the
proof of our results concerning nonsmooth potentials, including the proof of
our main existence result.



\section{Notations}
\label{sec:Preliminaries}

Here, we specify general notations, definitions and conventions required for
the subsequent analysis.

\paragraph{General notations} 
For two vectors $a,b\in\R^3$ we denote their inner product by
$a\cdot b=a_j b_j$ and their tensor product by $a\otimes b$ with
$\np{a\otimes b}_{jk}=a_jb_k$.  Here and in the following we use Einstein's 
summation convention and implicitly sum over repeated indices from $1$ to $3$.
The inner product of two tensors $A,B\in\R^{3\times3}$ is denoted by
$A:B=A_{jk}B_{jk}$.  Moreover, $A^\transpose$ and $\trace A$ denote the
transpose and the trace of $A$.  We further set
$a\otimes b:A=\np{a\otimes b}:A=a_j A_{jk}b_k$ and, if $C\in\R^{3\times3}$ is a
third tensor, $AB:C=(AB):C=A_{jk}B_{k\ell}C_{j\ell}$.

Usually, $\Omega\subset\R^3$ is a bounded Lipschitz
domain and $T\in(0,\infty]$.  Points $(x,t)$ in the space-time cylinder $\OmegaT$, consist of a
spatial variable $x\in\Omega$ and a time variable $t\in(0,T)$.  For a
sufficiently regular function $\uvel$, we denote its partial derivatives in
time and space by $\pt\uvel$ and $\partial_j\uvel$, $j=1,2,3$, respectively.
The symbols $\grad$ and $\Delta$ denote (spatial) gradient and Laplace
operator.  If $\vvel$ is a vector-valued function, we let
$\Div\vvel=\partial_j\vvel_j$ denote its divergence and set
$\vvel\cdot\grad\uvel=\vvel_j\partial_j\uvel$.  Symmetric and antisymmetric
parts of $\grad\vvel=(\partial_k\vvel_j)$ are given by
\[
\straintensor(\vvel)\coloneqq \half\np{\grad\vvel+\grad\vvel^\transpose}, 
\qquad
\rottensor(\vvel)\coloneqq\half\np{\grad\vvel-\grad\vvel^\transpose},
\]
respectively.
If $\Stens$ is a tensor-valued function,
its divergence $\Div\Stens$ is given by 
$(\Div\Stens)_j=\partial_k\Stens_{jk}$.
If $\Ttens$ is another tensor-valued function,
we define $\grad\Ttens:\grad\Stens=\partial_\ell\Ttens_{jk}\partial_\ell\Stens_{jk}$
and $v\cdot \nabla \Ttens : \Stens = v_j (\partial_j\Ttens_{k\ell})
\Stens_{k\ell}$.

\paragraph{Function spaces} 
Let $k\in\N_0\cup\set{\infty}$ and $A\in\set{\Omega,\overline{\Omega}}$.  Then
the class $\CR{k}\np{A}$ consists of all $k$-times continuously differentiable
(real-valued) functions on $A$, and $\CRc{k}\np{A}$ contains all compactly
supported functions in $\CR{k}\np{A}$.  By $\LR{q}(\Omega)$ with
$q\in[1,\infty]$ we denote the classical Lebesgue spaces with corresponding
norm $\norm{\cdot}_{q}$, and $\HSR{k}(\Omega)$ with $k\in\N$ denotes the
$\LR{2}$-based Sobolev space of order $k$, equipped with the norm
$\norm{\cdot}_{k,2}$.  
Moreover, $\HSRN{1}(\Omega)$ contains all elements of
$\HSR{1}(\Omega)$ with vanishing boundary trace, and
$\HSR{1/2}(\partial\Omega)$ denotes the class of boundary traces of functions
from $\HSR{1}(\Omega)$.  By $\HSR{-1}(\Omega)$ and $\HSR{-1/2}(\partial\Omega)$
we denote the dual spaces of $\HSRN{1}(\Omega)$ and
$\HSR{1/2}(\partial\Omega)$, respectively, where we use the distributional
duality pairing. 

The norm of a Banach space $X$ is denoted by $\norm{\cdot}_{X}$, and the same
symbol is used for the norms of $X^3$ and $X^{3\times3}$. When the dimension is
clear from the context, we simply write $X$ instead of $X^3$ or $X^{3\times3}$.
Moreover, $X'$ denotes the dual space of $X$, and $\CR{1,1}(X)$ is the
set of all continuously Fr\'echet differentiable functions $X\to\R$ with
globally Lipschitz continuous derivative.

For an interval $I\subset\R$, the class $\CR{0}(I;X)$ consists of all
continuous $X$-valued functions, and $\CRweak(I;X)$ consists of all weakly
continuous $X$-valued functions.  The Bochner--Lebesgue spaces of $X$-valued
functions are denoted by $\LR{q}(I;X)$ for $q\in[1,\infty]$, and
$\LRloc{q}(I;X)$ denotes the class of all functions that belong to
$\LR{q}(J;X)$ for all compact subintervals $J\subset I$.  When $I=(0,T)$, we
set $\CR{0}\np{0,T;X}=\CR{0}\np{I;X}$ and
$\LR{q}\np{0,T;X}=\LR{q}\np{I;X}$. For functions $A$ on $\Omega \times I$ we
use the shorthand $A(t) \coloneqq A(\,\cdot\,,t)$. 

We further need spaces of solenoidal vector fields and of symmetric deviatoric
tensor fields.  The corresponding classes of smooth functions on $\Omega$ are
given by
\[
\begin{aligned}
\CRcisigma\np{\Omega}
&\coloneqq\setcl{\varphi\in\CRci(\Omega)^3}{\Div\varphi=0},
\\
\CRidev\np{\overline{\Omega}}
&\coloneqq\setcl{\psi\in\CRi(\overline{\Omega})^{3\times3}}
{\psi=\psi^\transpose,\ \trace\psi=0}.
\end{aligned}
\]
We further set
\[
\begin{aligned}
\CRcisigma\np{\Omega \tim I}
&\coloneqq\setcl{\Phi\in\CRci(\Omega\tim I)^3}{\Div\Phi=0},
\\
\CRcidev\np{\overline{\Omega}\tim I}
&\coloneqq\setcl{\Psi\in\CRci(\overline{\Omega}\tim I)^{3\times3}}
{\Psi=\Psi^\transpose,\ \trace\Psi=0},
\end{aligned}
\]
and we set $\CRidev\np{\overline{\Omega}\tim I}\coloneqq\CRcidev\np{\overline{\Omega}\tim I}$ 
if $I$ is a compact interval.
We define the associated $\LR{2}$ spaces on $\Omega$ by
\[
\begin{aligned}
\LRsigma{2}(\Omega)
&\coloneqq\setcl{\vvel\in\LR{2}(\Omega)^3}{\Div\vvel=0, \ 
\restriction{\vvel}{\partial\Omega}\cdot\nvec=0}
=\closure{\CRcisigma\np{\Omega}}{\norm{\cdot}_{2}},
\\
\LRdev{2}(\Omega)
&\coloneqq\setcl{\Stens\in\LR{2}(\Omega)^{3\times3}}
{\Stens=\Stens^\transpose,\ \trace\Stens=0}
=\closure{\CRidev\np{\overline{\Omega}}}{\norm{\cdot}_{2}}.
\end{aligned}
\]
Here the conditions $\Div\vvel=0$ and $\restriction{\vvel}{\partial\Omega}\cdot\nvec=0$
in the definition of $\LRsigma{2}(\Omega)$
have to be understood in a weak sense; 
see \cite[Theorem III.2.3]{GaldiSteadyStateNS_2011} for example.
We further introduce the corresponding Sobolev spaces
\[
\begin{aligned}
\HSRNsigma{1}(\Omega)
&\coloneqq\setcl{\vvel\in\HSRN{1}(\Omega)^3}{\Div\vvel=0}
=\closure{\CRcisigma\np{\Omega}}{\norm{\cdot}_{1,2}},
\\
\HSRdev{1}(\Omega)
&\coloneqq\setcl{\Stens\in\HSR{1}(\Omega)^{3\times3}}
{\Stens=\Stens^\transpose,\ \trace\Stens=0}
=\closure{\CRidev\np{\overline{\Omega}}}{\norm{\cdot}_{1,2}}.
\end{aligned}
\]
We can now define the solution spaces
\[
\LH
\coloneqq
\LR{\infty}\np{0,T;\LRsigma{2}(\Omega)}
\cap\LR{2}\np{0,T;\HSR{1}(\Omega)^3}
\]
for the fluid velocity and
\[
\XT
\coloneqq
\LR{\infty}\np{0,T;\LRdev{2}(\Omega)}
\cap\LR{2}\np{0,T;\HSR{1}(\Omega)^{3\times3}}
\] 
for the stress tensor.
Observe that $\LH$ is the classical Leray--Hopf class for 
weak solutions to the Navier--Stokes equations,
and $\XT$ is the analog for semilinear parabolic equations taking values in
deviatoric tensor fields.
Moreover, for the introduction of the variational formulation of \eqref{sys:PDE} below, 
the space
\begin{align}\label{eq:testspace}
	\TestFn:=\HSR{1}(0,T';\LR{2}_\delta(\Omega))
	\cap \LR{2}(0,T';\HSR{1}(\Omega)^{3\times3})\cap \LR{5}(0,T';\LR{5}(\Om)^{3\times3})
\end{align}
will serve as the class of test functions.

\paragraph{Convex subdifferential}\label{page:disspot}
Let $\potential\colon\hs\to[0,\infty]$ be convex with $\potential(0)=0$. 
We denote by $\partial\mathcal{P}$ the convex subdifferential of $\mathcal{P}$, i.e., for $S\in \hs$ we let 
\[
  \partial\mathcal{P}(S):=\setcl{\tau\in \hs }{  (\tau,\tilde S -S)_\hs+\mathcal{P}(S)\le
  \mathcal{P}(\tilde S)\;\text{ for all }\tilde S\in \hs }.
\]
Observe that, by definition,
$\partial\mathcal{P}(S)=\emptyset$ if $\potential(\Stens)=+\infty$.  If
$\partial\mathcal{P}(S)=\{\tau\}$ for some $\tau\in \hs$, we identify the set
$\partial\mathcal{P}(S)$ with its unique element $\tau$. In this case, we call
$\tau$ the (G\^ateaux) differential of $\mathcal{P}$ at $S$.

\section{Main results}\label{sec:main.results}

The main contribution of our analysis to the large-data existence theory for
rate-type viscoelastoplastic fluid models lies in its ability to deal with
\textit{nonsmooth} dissipation potentials $\mathcal{P}$ under the mild
hypothesis that
\begin{align}\label{eq:propP}
	\potential\colon\LRdev{2}(\Om)\to[0,\infty] 
	\text{ 
		is convex and lower semicontinuous with $\potential(0)=0$.}
\end{align}
Throughout this paper, we let $T\in(0,\infty]$ and impose the following
conditions on the initial data and external forcing
\begin{equation}\label{el:data}
\begin{aligned}
&\Vvel_0\in\LRsigma{2}(\Omega), \quad
\Stens_0\in\LRdev{2}(\Omega), &&
F=F_0+\Div F_1,
\\
&F_0\in\LRloc{1}\np{[0,T);\LR{2}(\Omega)^3},
&&
F_1\in\LRloc{2}\np{[0,T);\LR{2}(\Omega)^{3\times3}}.
\end{aligned}
\end{equation}
In our main results we focus on boundary data $g$ that satisfy
\begin{equation}\label{el:cond_g}
g\in\LR{\infty}\np{0,T;\HSR{1/2}(\partial\Omega)^3},\quad
\pt g\in\LR{\infty}\np{0,T;\HSR{-1/2}(\partial\Omega)^3}
\end{equation}
and $g\cdot\nvec=0$ on $\partial\Omega\times(0,T)$.

Section \ref{su:Modeling} will be devoted to the motivation for allowing
for nonsmooth and set-valued stress-strain relations $\Stens \mapsto \partial 
\mathcal{P}(\Stens)$ and non-trivial boundary data $g$ by discussing
applications in the geodynamics of lithospheric plate motion.

\subsection{Generalized solution concept}\label{ssec:gen.sc}
For nonsmooth potentials $\mathcal{P}$ the subdifferential $\partial\mathcal{P}$ may be multi-valued, and hence, line~\eqrefsub{sys:PDE}{3} cannot be replaced with an equality but rather has to be understood as a suitable inclusion. 
Our notion of solution for problem~\eqrefsub{sys:PDE}{3} 
adapts the weak solution concept in~\cite[Chapter~10]{Roubicek_2013} for semilinear parabolic equations with nonsmooth potentials, which involves an evolutionary variational inequality that avoids the multi-valued function $\partial \mathcal{P} (S)$
by exploiting the inequality
\begin{align}\label{eq:554}
	\int_\Om\partial \mathcal{P} (S):(\tilde S{-}S)\,\dd x\leq
\mathcal{P} (\tilde S) - \mathcal{P} (S)
\end{align} 
for suitably regular test function $\tilde S$. (If $\partial \mathcal{P} (S)$ is multi-valued, the left-hand side of the last inequality should be understood elementwise, replacing $\partial \mathcal{P} (S)$ by $\beta\in\partial \mathcal{P} (S)$.)
In comparison to~\cite[Chapter~10]{Roubicek_2013}, the analysis of the present problem is, however, greatly complicated by the presence of the geometric nonlinearities in~\eqrefsub{sys:PDE}{3} (coming from the objective derivative $\jaumannder{S}$) as well as the coupling term $\eta_2 D(V)$.

To motivate the evolutionary variational inequality we are going to propose for~\eqrefsub{sys:PDE}{3}, 
let us assume for the moment smoothness of the functions involved and suppose that~\eqrefsub{sys:PDE}{3} holds as an equality.
%
Multiplying this equation by $(\tilde S-S)$ and integrating over $\Om$ then allows us to use ineq.~\eqref{eq:554} and avoid $\partial \mathcal{P} (S)$ at the expense of an inequality.
The integral involving the Zaremba--Jaumann derivative can be modified via the
identity \eqref{eq:Jaum.Identity},
\begin{equation}\label{eq:208}
	\begin{aligned}
\int_\Omega \jaumannder{S} :(\tilde S{-}S) \dx &= \int_\Omega \!\Big(
\partial_t S:(\tilde S{-}S) + \big(\jaumannder{S}{-} \partial_t S):\tilde S
\Big)\dx\\
& = \int_\Omega \!\Big(
\partial_t \tilde S:(\tilde S{-}S) + \big(\jaumannder{S}{-} \partial_t S):\tilde S
\Big)\dx - \frac{\mathrm d}{\mathrm d t} \int_\Omega \half|\tilde S{-}S|^2
\dx ,
\end{aligned}
\end{equation}
where $\jaumannder{S}- \partial_t S= V\cdot \grad S + SW-WS$ 
and where we have used the fact that $\int_\Om (\jaumannder{S}{-} \partial_t S): S
\dx=0$ for $(V,S)$ is sufficiently regular.
 Upon integration in time from $t=0$ to $t=T'<T$, we then arrive at 
\begin{align*}
	\begin{split}
		&\int_0^{T'}\!\!\!\int_\Omega \pt\tilde\Stens:(\tilde S-\Stens)
		+\gamma\grad\Stens:\grad(\tilde S-\Stens)\,\dd x\dd t +
		\int_0^{T'}\!\!\Big(\potential(\tilde\Stens)-\potential(\Stens)\Big)\dt
		\\
		&+\int_0^{T'}\!\!\!\int_\Omega V\cdot\grad\Stens:\tilde S
		+(\Stens\rottensor\np{V}-\rottensor\np{V}\Stens):\tilde S -\eta\straintensor(\Vvel):(\tilde S-\Stens)\dx\dt \\
		&\qquad\qquad\qquad\qquad\qquad\qquad\qquad
		\geq \tfrac{1}{2}\|\tilde S(T')-\Stens(T')\|_{2}^2
		-\tfrac{1}{2}\|\tilde S( 0)-\Stens_0\|_{2}^2.
	\end{split}
\end{align*}
To avoid issues when passing to the limit along approximate solutions, 
we follow~\cite[Chapter 10]{Roubicek_2013} and drop the positive term
$\tfrac{1}{2}\|\tilde S(T')-\Stens(T')\|_{2}^2$ on the right-hand side in our
ultimate variational inequality for~\eqrefsub{sys:PDE}{3} (cf.\
eq.~\eqref{eq:varin} below). 

\begin{defn}[Generalized solution]\label{def:varsol} 
Let $\potential$ satisfy~\eqref{eq:propP}.
  We call a couple $(V,S)$ a \emph{generalized solution to \eqref{sys:PDE}--\eqref{eq:InitialConditions}}
  if for all $T'\in(0,T)$ the following
  holds: $(V,S)\in \LHprime\times\XTprime$ satisfies $\restriction{V}{\partial\Om\times(0,T)}=g$
  as well as
  \begin{align}
  \begin{split}
	&\int_0^T\int_\Omega \bb{
	-\Vvel\cdot\pt\Phi
	-\Vvel\otimes\Vvel:\grad\Phi
	+\eta \Stens:\grad\Phi
	+\mu\grad\Vvel:\grad\Phi
	}\dx\dt
	\\
	&\qquad
	=\int_0^T\int_\Omega F_0\cdot\Phi\dx\dt
	-\int_0^T\int_\Omega F_1:\grad\Phi\dx\dt
	+\int_\Omega\Vvel_0\cdot\Phi(\cdot,0)\dx,
	\end{split}
	\label{eq:weaksol_V}
	  \end{align}
  for all $\Phi\in\CRcisigma(\Om\times[0,T))$
   and 
	  \begin{align}
    \begin{split}
      &\int_0^{T'}\!\!\!\int_\Omega \pt\tilde\Stens:(\tilde S-\Stens)
      +\gamma\grad\Stens:\grad(\tilde S-\Stens)\,\dd x\dd t +
      \int_0^{T'}\!\!\Big(\potential(\tilde\Stens)-\potential(\Stens)\Big)\dt
      \\
      &+\int_0^{T'}\!\!\!\int_\Omega V\cdot\grad\Stens:\tilde S
      +(\Stens\rottensor\np{V}-\rottensor\np{V}\Stens):\tilde S -\eta\straintensor(\Vvel):(\tilde S-\Stens)\dx\dt \\
      &\qquad\qquad\qquad\qquad\qquad\qquad\qquad\qquad\qquad\qquad\qquad
      \geq-\tfrac{1}{2}\|\tilde S( 0)-\Stens_0\|_{2}^2
    \end{split}
    \label{eq:varin}
  \end{align}
  for 
  all $\tilde S \in \TestFn$.
\end{defn}

Observe that \eqref{eq:weaksol_V} is obtained by multiplying \eqrefsub{sys:PDE}{1} 
by the respective test functions 
and formally integrating by parts.
In particular, \eqref{eq:weaksol_V} is  
in accordance with
the notion of weak solutions for the classical Navier--Stokes problem,
since we take divergence-free test functions
and omit the pressure term.
Moreover, it is easy to see that the terms in~\eqref{eq:varin} are well-defined. 
For the integrals involving the nonlinear terms of the Zaremba--Jaumann rate, 
this follows from the Sobolev embedding $\HSR{1}(\Om)\hookrightarrow \LR{6}(\Om)$, the interpolation   
\begin{align}\label{eq:interpol}
	\LR{\infty}(0,T';\LR{2}(\Om))\cap \LR{2}(0,T';\LR{6}(\Om))\hookrightarrow \LR{\frac{10}{3}}(0,T';\LR{\frac{10}{3}}(\Om)),
\end{align}
and the generalized H\"older inequality with inverse exponents $\frac{3}{10}+\frac{1}{2}+\frac{1}{5}=1$.

\begin{rem}  In the formulation~\eqref{eq:varin} it is crucial that in the integral
		involving the convective part, only the term $V\cdot\nabla S:\tilde S$
		occurs, and not $V\cdot\nabla S:(\tilde S-S)$, since under the natural
		regularity hypotheses of~$S$ in Def.~\ref{def:varsol}, integrability of the term $V\cdot\nabla S:S$ is not
		ensured.
\end{rem}

It is worth noting that, despite the absence of the term $\tfrac{1}{2}\|\tilde S(T')-\Stens(T')\|_{2}^2$ in ineq.~\eqref{eq:varin}, generalized solutions in the sense of Def.~\ref{def:varsol} obey  the natural partial energy dissipation inequality for $S$.

\begin{prop}[Partial energy inequality]\label{prop:edin.genP}
	Any generalized solution $(V,S)$ in the sense of Definition~\ref{def:varsol} satisfies the partial energy dissipation inequality
	\begin{equation}\label{eq:edin.S.genP}
		\begin{aligned}
			\half\| S(T')\|_2^2 &+ \gamma \norm{\grad\Stens}_{\LR{2}\np{\Omega\times(0,T')}}^2
			+\int_0^{T'}\!\!\potential(\Stens)\dtau 
			\\
			&\qquad\qquad
			\leq\half\norm{\Stens_0}_2^2
			+\int_0^{T'}\!\!\int_\Omega \eta\straintensor(\Vvel):\Stens\dx\dtau
		\end{aligned}
	\end{equation}
for almost all $T'\in(0,T)$. 
\end{prop}
Observe that this proposition applies to any solution conforming to Def.~\ref{def:varsol} and
is independent of the approximation scheme chosen in our construction. Its proof will therefore be postponed to Section~\ref{ssec:pf.eni}.

\subsection{Main results}

The main achievement of this article is the following result,
which shows global-in-time existence of generalized solutions
in the sense of Definition \ref{def:varsol}.

\begin{thm}[Existence of generalized solutions
  to~\eqref{sys:PDE}--\eqref{eq:InitialConditions}]\label{thm:genP}
  Let $T\in(0,\infty]$, let $\Omega$ be a bounded domain in $\R^3$ with
  $\CR{1,1}$-boundary, and let $\potential$ satisfy~\eqref{eq:propP}. Let $\Vvel_0$,
  $\Stens_0$ 
  and $F$ be as in \eqref{el:data}, and let $g$ satisfy
  \eqref{el:cond_g} and $g\cdot\nvec=0$ on $\partial\Omega\times(0,T)$.  Then there exists a generalized
  solution $(V,S)$ 
  to \eqref{sys:PDE}--\eqref{eq:InitialConditions}
  in the sense of Definition~\ref{def:varsol}.
  
  Moreover, there exists an extension $\wvel$ of the boundary data $g$ 
  such that $v=V{-}w$ 
  satisfies the following energy-dissipation inequality: 
  For a.a.~$t\in (0,T)$ we have 
  \begin{equation}
\begin{aligned}
&\int_\Om \frac12|v(t)|^2
+\int_0^t \!\!\int_\Om
\mu |\nabla v|^2  \dx \dtau
\\
&
\ \leq\int_\Om \frac12|V_0{-}w(0)|^2 \dx
\\
&\quad
+\int_0^t\! \int_{\Omega}\bb{F_0\cdot\vvel
-(F_1{-}\widetilde F_1):\grad\vvel
+\wvel\otimes\np{\vvel{+}\wvel}:\grad\vvel
-\eta\Stens:\grad\vvel} \dx\dtau,
\end{aligned}
\label{est:EnergyInequality_v_combined}
\end{equation}
where $\widetilde F_1$ is determined by
$\Div \widetilde F_1=\pt\wvel-\mu\Delta\wvel$.
In particular, for a.a.~$t\in(0,T)$ we have the total energy-dissipation inequality
\begin{align}
\nonumber
&\int_\Om \frac12|v(t)|^2 + \frac12|S(t)|^2 \dx + \int_0^t \!\!\int_\Om \Big[
\mu |\nabla v|^2 + \gamma |\nabla S|^2\Big] \dx \dtau
+ \int_0^t \mathcal P(S) \dtau\\ 
\label{eq:GenSol.EDI}
&\quad \leq \int_\Om \frac12|V_0{-}w(0)|^2 + \frac12|S_0|^2 \dx \\
\nonumber
&\qquad + \int_0^t
\!\! \int_\Om \Big[F_0\cdot v - \np{F_1{-}\widetilde F_1}:\nabla v + w\otimes(v{+}w):\nabla v + \eta
D(w): S  \Big] \dx \dtau .
\end{align}
The function $\wvel$ can be chosen such that
$\wvel\in\LR{\infty}\np{0,T;\HSR{1}(\Omega)^3}$,
$\pt \wvel\in\LR{\infty}\np{0,T;\HSR{-1}(\Omega)^3}$,
and such that $\wvel$ satisfies the estimates
\begin{equation}\label{est:w.nonlin.mr}
\int_0^t\int_\Omega
\wvel\otimes\vvel:\grad\vvel\dx\dtau
\leq \frac{\mu}{2} \norm{\grad\vvel}_{\LR{2}\np{\Omegat}}^2,
\end{equation}
and
\begin{equation}\label{est:w.lin}
\begin{aligned}
\norm{\wvel}_{\LR{q}(0,T';\HSR{1}(\Omega))}
&\leq C\norm{g}_{\LR{q}(0,T';\HSR{1/2}(\partial\Omega))},
\\
\norm{\pt\wvel}_{\LR{q}\np{0,T';\HSR{-1}(\Omega)}}
&\leq C\norm{\pt g}_{\LR{q}\np{0,T';\HSR{-1/2}(\partial\Omega)}}
\end{aligned}
\end{equation} 
for all $T'\in(0,T)$ and $q\in[1,\infty]$ and a suitable constant $C=C(q,\Omega,\mu)>0$.
\end{thm}

Note that \eqref{eq:GenSol.EDI} directly follows from summation of \eqref{est:EnergyInequality_v_combined}
and \eqref{eq:edin.S.genP},
which holds by Proposition \ref{prop:edin.genP},
and using $S:\grad\vvel=S:\straintensor(\vvel)$.
One readily verifies that all terms in \eqref{eq:GenSol.EDI}
are well defined when $\wvel$ has the stated regularity.
Moreover, for $g=0$ we have $\wvel=0$,
in which case the right-hand side of the 
energy-dissipation inequality \eqref{eq:GenSol.EDI}
simplifies significantly.
Another consequence is the following result,
that gives a class of data
such that the total energy 
stays bounded as $t\to\infty$.

\begin{cor}\label{cor:GlobalSolutionsEnergy}
In the situation of Theorem \ref{thm:genP},
let $T=\infty$ and assume
\[
\begin{aligned}
F_0&\in\LR{1}\np{0,\infty;\LR{2}(\Omega)},
\qquad
F_1\in \LR{2}(0,\infty;\LR{2}(\Omega)),
\\
g&\in\LR{\infty}\np{0,\infty;\HSR{1/2}(\partial\Omega)}
\cap\LR{1}\np{0,\infty;\HSR{1/2}(\partial\Omega)},
\\
\pt g&\in\LR{\infty}\np{0,\infty;\HSR{-1/2}(\partial\Omega)}
\cap\LR{2}\np{0,\infty;\HSR{-1/2}(\partial\Omega)}.
\end{aligned}
\]
Then $\np{\Vvel,\Stens}\in\LH\times\XT$ for $T=\infty$, 
so that the total energy remains bounded as $t\to\infty$.
\end{cor}

The proof of both Theorem \ref{thm:genP} and Corollary \ref{cor:GlobalSolutionsEnergy}
will be given in Subsection~\ref{ssec:proof.main}.

\subsection{Strategy of the proof and existence for smooth potentials}\label{ssec:strategy}

\paragraph{Moreau envelope} The proof of  Theorem \ref{thm:genP} is based on approximating the nonsmooth potential $\potential$
by its Moreau envelope 
\begin{equation}
	\label{eq:YosMorEnvel}
	\mathcal{P}_\ve(S)=\inf_{S'\in
		\LRdev{2}(\Omega)}\bigg(\frac{1}{2\ve}\|S-S'\|_{\LRdev{2}(\Omega)}^2+\mathcal{P}(S')\bigg),\quad\ve\in(0,1].
\end{equation}
This regularization preserves the basic properties~\eqref{eq:propP} imposed on our potentials, that is, for each $\ve\in(0,1]$ the regularized potential $\mathcal{P}_\ve:	\LRdev{2}(\Omega)\to[0,\infty)$ is convex and lower semicontinuous with $\potential_\ve(0)=0$. 
Furthermore, each $\mathcal{P}_\ve$ is Fr\'echet differentiable and its differential 
$\partial\mathcal{P}_\ve$ is globally Lip\-schitz continuous 
with Lipschitz constant $1/\varepsilon$
(see~\cite[Section~12.4]{BC_2017} for example).  
Since $\Stens=0$ is a minimum of $\potential_\varepsilon$ and $\potential_\ve(0)=0$, 
the Lipschitz continuity of $\partial\mathcal{P}_\ve$ implies that
\[
\norm{\partial\potential_\varepsilon(\Stens)}_{\LRdev{2}(\Omega)}
=\norm{\partial\potential_\varepsilon(\Stens)-\partial\potential_\varepsilon(0)}_{\LRdev{2}(\Omega)}
\le \varepsilon^{-1}\norm{S}_{\LRdev{2}(\Omega)}.
\]
%
%
\paragraph{Weak solutions for smooth potentials}
In a first step of our analysis, we provide an existence result to system~\eqref{sys:PDE}--\eqref{eq:InitialConditions}
for smooth potentials $\potential\in\CR{1,1}(\LRdev{2}(\Omega))$
allowing us to use a standard concept of weak solutions. 

%
%

\begin{defn}[Weak solutions]\label{def:weaksol}
Let $\potential\in\CR{1,1}(\LRdev{2}(\Omega))$.
We call a couple $\np{\Vvel,\Stens}$ a 
\emph{weak solution to \eqref{sys:PDE}--\eqref{eq:InitialConditions}}
if
$\np{\Vvel,\Stens}\in\LHprime\times\XTprime$
for all $0<T'< T$,
if $\restriction{\Vvel}{\partial\Omega\times(0,T)}=g$,
and if the identities \eqref{eq:weaksol_V} and
\begin{equation}
\begin{aligned}
\int_0^T\!\!\int_\Omega &\bb{
-\Stens:\pt\Psi
+\Vvel\cdot\grad\Stens:\Psi
+\Stens\rottensor\np{\Vvel}:\Psi
-\rottensor\np{\Vvel}\Stens:\Psi
\\
&
+\partial\potential(\Stens):\Psi
+\gamma\grad\Stens:\grad\Psi
-\eta\grad\Vvel:\Psi
}\dx\dt
=\int_\Omega\Stens_0:\Psi(\cdot,0)\dx
\end{aligned}
\label{eq:weaksol_S}
\end{equation}
hold for all $\Phi\in\CRcisigma\np{\OmegaTzero}$ and $\Psi\in\CRcidev(\OmegaTzerocl)$.
\end{defn}


Our existence result for $\CR{1,1}(\LRdev{2}(\Omega))$-smooth potentials states as follows.
\begin{thm}\label{thm:existence_weaksol}
In addition to the hypotheses of Theorem \ref{thm:genP}, assume 
$\potential\in\CR{1,1}(\LRdev{2}(\Omega))$.
Then
there exists a weak solution $\np{\Vvel,\Stens}$ 
to \eqref{sys:PDE}--\eqref{eq:InitialConditions}
in the sense of Definition \ref{def:weaksol}.
This solution is weakly continuous in $\LR{2}(\Omega)$ in the sense that
\begin{equation}
  \label{eq:WeakContinuity_VS}
  \Vvel\in\CRweak(0,T;\LRsigma{2}(\Omega)), 
  \qquad \Stens\in\CRweak(0,T;\LRdev{2}(\Omega)), 
\end{equation}
with $\bp{\Vvel(0),\Stens(0)}=\bp{\Vvel_0,\Stens_0}$. 
Moreover, $\Vvel$ can be decomposed as $\Vvel=\vvel+\wvel$, 
where $\wvel$ is the same extension of $g$ as in Theorem \ref{thm:genP},
such that for all $t\in(0,T)$ we have the partial energy-dissipation inequalities
\eqref{est:EnergyInequality_v_combined} and
\begin{equation}
\begin{aligned}
&\half\norm{\Stens(t)}_2^2
+\gamma\norm{\grad\Stens}_{\LR{2}(\Omegat)}^2
+\int_0^t\int_\Om\partial\potential(\Stens):\Stens\,\dd x\dtau 
\\
&\qquad\qquad\qquad\qquad\qquad\qquad\qquad
\leq\half\norm{\Stens_0}_2^2 
+\int_0^t\int_\Omega \eta\straintensor(\Vvel):\Stens\dx\dtau.
\end{aligned}
\label{est:EnergyInequality_S_combined}
\end{equation}
\end{thm}
Observe that the stress tensor $S$ obtained in Theorem~\ref{thm:existence_weaksol} enjoys better regularity properties than that in Theorem \ref{thm:genP}. 
In particular, we note that the fact that $(V,S)$  satisfies the tensor evolution problem in the weak sense implies an estimate on the time derivative $\partial_tS$ in $\LR{8/7}\np{0,T';\np{\HSR{1}(\Omega)}'}$ (cf.\ Remark~\ref{rem:weaksol_PDEmodified_bounds})
as well as the weak continuity of $t\mapsto S(t)$ in $\LRdev{2}(\Omega)$. 
Further note that 
\begin{equation}\label{est:dPS}
\int_\Omega\partial\potential(\Stens):\Stens\dx
=\int_\Omega\partial\potential(\Stens):(\Stens-0)\dx
\geq \potential(\Stens)-\potential(0)=\potential(\Stens).
\end{equation}
Hence, the partial energy-dissipation inequality \eqref{eq:edin.S.genP}, satisfied by any generalized solution due to Proposition \ref{prop:edin.genP},
may in general be weaker than the corresponding inequality~\eqref{est:EnergyInequality_S_combined}, which the weak solution constructed in Theorem~\ref{thm:existence_weaksol} conform to.

Our construction of solutions for smooth $\potential$ is based on a Galerkin
approximation that is manufactured in such a way that the global energy control in Corollary~\ref{cor:GlobalSolutionsEnergy} holds true for data with suitable time decay.
Exploiting the standard energy estimates and relying on compactness
arguments of Aubin--Lions type for $\Vvel$ and $\Stens$  
allows us to pass to the weak limit even in the nonlinear terms $\Vvel \cdot
\nabla\Stens$ and $\Stens   \rottensor(\Vvel)$. Of course, in the limit the
energy-dissipation balance \eqref{eq:401} will only survive as an
energy-dissipation inequality. 

When approaching the nonsmooth dissipation potentials $\potential$ by their
Moreau envelope $\potential_\ve$, we lose the compactness for
$\Stens^\ve$ because $\partial \potential_\varepsilon(\Stens^\ve)$
cannot be controlled. 
 Nevertheless, we are able to pass to the limit in the
critical terms of the form $\Stens^\ve \nabla \Vvel^\ve$ by integration by
parts and relying on the boundary conditions, see Lemma \ref{le:SW(V)}.

\subsection{Some remarks concerning the modeling}
\label{su:Modeling}


Our work is motivated by the modeling of geophysical flows, where rock is
considered as a fluid flowing on very long time scales with speeds of
millimeters per year, see
\cite{MoDuMu02MCMV,GerYue07RCMM,HerrenGeryaVand2017IRSDF,PHGAV2019SAFG}.  The
aim of those works is to understand the deformations of tectonic plates, the
formation and development of faults, and the simulation of aseismic slipping
processes.  The modeling often concerns a smaller fault area of weaker material
in the domain $\Omega$ that is driven by the more rigid tectonic plates around
the weaker material.  This situation is modeled in our case by prescribing
the velocity $g(t,\cdot)$ along the boundary $\partial \Omega$. 

The important common feature of these geodynamical models is the plasticity
threshold for the shear stress, which is defined 
 in terms of the norm of the deviatoric
stress tensor, namely $|\Stens|\leq \sigma_\text{yield}$, because of our
assumption $\trace (\Stens)=0$. This condition is mathematically formulated 
via dissipation potentials $\mathcal{P}$ of the form
$\potential (S)=\int_\Omega \mathfrak P(S(x))\dx$
with $\mathfrak P$ satisfying
$\mathfrak P( \Stens)=\infty$ for $|\Stens|> \sigma_\text{yield}$.

Indeed the evolution equation \eqref{sys:PDE}$_3$ for $\Stens$ has to be 
invariant under time-dependent changes of the observer (cf.\ \cite{Antm98PUVS})
which means that $\mathfrak P$ has to satisfy 
\[
\mathfrak P(Q\Stens Q^\top ) =  \mathfrak P( \Stens)  \quad \text{for all
} \Stens \in \R^{3\times 3}_\delta \text{ and } Q\in \text{SO}(\R^3).
\] 
Thus, $\mathfrak P$ can only depend on the three invariants of $\Stens \in
\R^{3 \times 3}_\text{sym}$, namely $\trace \Stens$,
$\trace(\Stens^2)-\trace(\Stens)^2$, and $\det\Stens$.  
Since we have restricted our analysis to the case $\trace \Stens=0$ (implying that $\trace(\Stens^2)-\trace(\Stens)^2=\snorm{\Stens}^2$), a very
typical choice is given in the form 
\[
\mathfrak P(\Stens)= \mathfrak p\big(|\Stens|\big) \quad \Longrightarrow \quad
\partial \mathfrak P(\Stens) = \partial \mathfrak p\big(|\Stens|\big)
\frac1{|\Stens|}\, \Stens,
\]
where $\mathfrak p:\R\to [0,\infty]$ is a lower semicontinuous, convex function
with $\mathfrak p(\sigma)=0$ for $\sigma\leq 0$. The choice
$\mathfrak p(\sigma) = \frac a2\sigma^2 $ for
$\sigma\in [0,\sigma_\text{yield}]$ and $\mathfrak p(\sigma)=\infty$ for
$\sigma>\sigma_\text{yield}$ gives the case in \eqref{eq:pot.example.intro}.

More general dissipation potentials may involve $\det \Stens$ as well. Using
that $\Stens\mapsto |\Stens|^6 +b (\det \Stens)^2$ is convex for
$|b|\leq b_*$ (e.g.\ $b_*=4$ suffices) 
and that $\frac{\partial \det
	\Stens}{\partial \Stens_{ij}}=\mathrm{cof}(\Stens)_{ij}$ 
(with $\mathrm{cof}(A)= \det(A)A^{-\top}$ for
invertible matrices) we may consider, for $a_2,a_4,a_6\geq 0$,  
\begin{align*}
\mathfrak P(\Stens)&= \frac{a_2}2 |\Stens|^2+ \frac{a_4}4 |\Stens|^4 +
\frac{a_6}6\big(|\Stens|^6+b(\det\Stens)^2
\big)  \\  \Longrightarrow& \;\;
\partial \mathfrak P(\Stens)= \big(a_2 {+}a_4|\Stens|^2{+} a_6|S|^4\big) S +
\frac{a_6b}3 \det(\Stens)\,\big( \text{cof}(\Stens) + \frac16|\Stens|^2 I\Big),
\end{align*}
where the last term was added to ensure $\partial\mathfrak P(\Stens)\in\R^{3\times 3}_\delta$,
and we used that $\trace \Stens =0$ implies
$\trace\big(\text{cof}(\Stens)\big)=- |\Stens|^2/2$. 
\medskip

In fact, without changing much in our analysis, we could have allowed the internal
stresses $\Stens$ to have a spherical part, i.e.\ we may have considered $\widehat \Stens
= \Stens + \widehat{p}  I \in \R^{3\times 3}_\text{sym}$ with $\Stens \in
\R^{3\times 3}_\delta$ and $\widehat{p}\in \R$. The analysis remains unchanged as
long as we keep the constraint $\Div V=0$ in \eqref{sys:PDE}$_2$.  This would
allow us to consider more general dissipation potentials allowing for a
coupling of $\Stens$ and $\widehat{p}$, e.g.\ in the form 
\[
\mathfrak P(\widehat \Stens) = \left\{ \begin{array}{cl} \frac a2 |\Stens|^2
    +\frac{b}2 \widehat{p}^2  & \text{for } 
    |S|\leq \sigma_\text{yield} + c\widehat{p}, \\ 
\infty & \text{for } |S|> \sigma_\text{yield}+ c\widehat{p},
\end{array}  \right. 
\]
where $a,b,c$ are positive material parameters. 

However, $\widehat{p}$ takes the position of an evolving additional pressure that
should be modeled together with the true pressure taking compressibility of the
fluid into account. Following
\cite{MoDuMu02MCMV,GerYue07RCMM,HerrenGeryaVand2017IRSDF} the
modeling assumption is then to identify $\widehat{p}$ and the pressure $P$ in
\eqref{sys:PDE}$_1$. Moreover, $P$ is then treated as an independent variable
solving an evolution equation like
$\frac1K \mathrm D_t P + \frac1\xi P =-\Div V$ for positive constants $K$ and $\xi$, which replaces
\eqref{sys:PDE}$_2$. The arising model is quite different from ours and will be
studied in subsequent work.

\section{Existence for smooth potentials}\label{sec:ex.smooth}

In this section we focus on the case of a smooth potential 
$\potential\in\CR{1,1}(\LRdev{2}(\Omega))$
and show existence of a weak solution to the system
\eqref{sys:PDE}--\eqref{eq:InitialConditions}
as stated in Theorem \ref{thm:existence_weaksol}.
To this end, we decompose the velocity field $\Vvel$ 
into a suitable extension of the boundary data $g$ 
and a remainder $\vvel$ with homogeneous boundary conditions
such that $\np{\vvel,\Stens}$ satisfies a perturbed system. 
While the existence of $\wvel$ will follow from well-known results,
we show existence of $\np{\vvel,\Stens}$ by a Galerkin method.
In the end, $\np{\Vvel,\Stens}=\np{\vvel+\wvel,\Stens}$
will be a weak solution with the properties asserted in Theorem \ref{thm:existence_weaksol}.

\subsection{Decomposition of the velocity field}
\label{sec:WeakSolPerturbed}

To show existence of a weak solution to the system
\eqref{sys:PDE}--\eqref{eq:InitialConditions}, we decompose the velocity and
pressure fields into two parts, $\Vvel=\vvel+\wvel$ and $\Vpres=\vpres+\wpres$,
where $\np{\wvel,\wpres}$ is a solution to the Stokes initial-value problem
with boundary data $w=g$ 
\begin{equation}
 \label{sys:StokesIVP}
\begin{pdeq}
\pt\wvel-\Div\bp{2\mu\straintensor(\wvel)-\wpres\idmatrix}
&=\widetilde{F} 
&&\tin\OmegaT, 
\\
\Div\wvel
&=0
&&\tin\OmegaT, 
\\
\wvel
&=g
&&\ton\partialOmegaT,
\\
\wvel(\cdot,0)
&=\wvel_0
&&\tin\Omega
\end{pdeq}
\end{equation}
for some functions $\widetilde{F}$ and $\wvel_0$,
and $\vvel$ satisfies the remaining system with homogeneous boundary data $v=0$ 
\begin{equation}
 \label{sys:PDE_perturbed}
\begin{pdeq}
\pt\vvel+\nsnonlin{\np{\vvel{+}\wvel}}-\Div 
 \bp{\eta\Stens+2\mu\straintensor(\vvel)-\vpres\idmatrix}
&=f
&&\tin \OmegaT, \\
\Div\vvel 
&=0
&&\tin \OmegaT, \\
\pt\Stens+\nsnonlinb{\np{\vvel{+}\wvel}}{\Stens}
+\Stens\rottensor\np{\vvel{+}\wvel}-\rottensor\np{\vvel{+}\wvel}\Stens 
\quad&
\\
{}+\partial\potential(\Stens)
-\gamma\Delta\Stens
-\eta \straintensor(\vvel{+}\wvel)
&=0
&&\tin\OmegaT,
\\
\vvel=0, \quad 
\nvec\cdot\grad\Stens
&=0
&&\ton\partialOmegaT,
\\
\vvel(\cdot,0)
=\vvel_0, \quad
\Stens(\cdot,0)
&=\Stens_0
&&\tin\Omega
\end{pdeq}
\end{equation}
with $f= F-\widetilde{F}$ and $\vvel_0=\Vvel_0-\wvel_0$.
In this way, we have decomposed the question of existence of solutions 
to \eqref{sys:PDE}--\eqref{eq:InitialConditions} into two 
separate problems. 

This decomposition method is a
common way to treat inhomogeneous boundary data $g\neq 0$, and was successfully
used to show existence of weak solutions to the classical Navier--Stokes
initial-value problem in different configurations, see
\cite{FarwigGaldiSohr_NewClassWeakSolNSNonhData,
  FarwigKozonoSohr_GlobalWeakSolNSNonzeroBdryCond_2010,
  FarwigKozonoSohr_GlobalLHWeakSolNSTimeDependentBdryData,
  FarwigKozonoSohr_GlobalWeakSolNSNonhBdryDataDiv_2011} for example.
Observe that this decomposition is by no means unique
since the functions $\widetilde F$ and $\wvel_0$ are not determined by 
the original problem \eqref{sys:PDE}--\eqref{eq:InitialConditions}.
This freedom allows us to construct $\wvel$ as a divergence-free extension of 
the boundary data $g$ satisfying~\eqref{est:w.nonlin.mr} and~\eqref{est:w.lin} and to specify $\tilde F, w_0$ afterwards.
%

\subsection{Weak solutions to the modified system}
In this section we prescribe a class of vector fields $\wvel$ such
that the modified system \eqref{sys:PDE_perturbed} has a solution, see
Assumption \ref{asm:extension} below. 
In particular, for the moment
we do not assume that the function $\wvel$ in \eqref{sys:PDE_perturbed} 
is related to some given boundary function $g$.

For showing the existence of weak solutions to the modified system
\eqref{sys:PDE_perturbed} we make the following assumptions. Let
$0<T\leq\infty$ and $\Omega\subset\R^3$ be a bounded Lipschitz domain.  Let
the dissipation potential $\potential$ satisfy \eqref{eq:propP} 
and $\potential\in\CR{1,1}(\LRdev{2}(\Omega))$.  
Similarly to \eqref{el:data}, for the data we
assume
\begin{equation}\label{el:data_perturbed}
\begin{aligned}
&\vvel_0\in\LRsigma{2}(\Omega), \quad
\Stens_0\in\LRdev{2}(\Omega),
&&
f=f_0+\Div f_1,
\\
&f_0\in\LRloc{1}\np{[0,T);\LR{2}(\Omega)^3},
&&
f_1\in\LRloc{2}\np{[0,T);\LR{2}(\Omega)^{3\times3}}.
\end{aligned}
\end{equation}
Moreover, the function $\wvel$ is assumed to 
have the following properties.

\begin{asm}\label{asm:extension}
The function $\wvel$ satisfies
\[
\wvel\in\LRloc{4}\np{[0,T);\LR{4}(\Omega)^3}, 
\qquad
\grad\wvel\in\LRloc{2}\np{[0,T);\LR{2}(\Omega)^{3\times3}},
\]
and one of the following three properties:
\begin{enumerate}[label=(\alph*)]
\item
$\wvel\in\LRloc{s}\np{[0,T);\LR{r}(\Omega)^3}$
for some $r\in(3,\infty)$, $s\in(2,\infty)$
with $\frac{2}{s}+\frac{3}{r}=1$,
\label{item:extension_i}
\item
$\wvel\in\CR{0}\np{0,T;\LR{3}(\Omega)^3}$
with $\norm{\wvel}_{\LR{\infty}\np{0,T;\LR{3}(\Omega)}}\leq\alpha$ for $\alpha>0$ sufficiently small,
\label{item:extension_ii}
\item
$\wvel\in\CR{0}\np{0,T;\LR{3}(\Omega)^3}$
and for all $t\in (0,T)$ we have 
\begin{align}
\forall\, \vvel\in\HSRNsigma{1}(\Omega):\quad
\snormL{\int_\Omega 
\wvel(t)\otimes\vvel:\grad\vvel
\dx}
&\leq\frac{\mu}{2}
\norm{\grad\vvel}_{2}^2,
\label{est:nonlinearterm_smallness_v}
\\
\forall\, \Stens\in\HSRdev{1}(\Omega):\quad
\snormL{\int_\Omega 
\wvel(t)\cdot\grad\Stens:\Stens
\dx}
&\leq\frac{\gamma}{2}
\norm{\Stens}_{1,2}^2,
\label{est:nonlinearterm_smallness_S}
\end{align}
where $\mu$ and $\gamma$ are the constants appearing in \eqref{sys:PDE}. 
\label{item:extension_iii}
\end{enumerate}
\end{asm}
Observe that in our main results in Theorem \ref{thm:genP} and Theorem \ref{thm:existence_weaksol}
we consider boundary data $g$ satisfying \eqref{el:cond_g}.
As we will see in Subsection \ref{subsec:extension},
this suffices to construct an extension $\wvel$ of $g$ 
satisfying Assumption \ref{asm:extension}\ref{item:extension_iii}. 
However, since there is not much difference in the proof,
we shall establish solutions to \eqref{sys:PDE_perturbed} 
in all cases described in Assumption \ref{asm:extension}.

\begin{rem}\label{rem:extension}
  Note that condition \ref{item:extension_i} cannot directly be generalized to
  the case $r=3$, $s=\infty$, which is treated in \ref{item:extension_ii} and
  \ref{item:extension_iii}.  Moreover, the smallness of the extension $\wvel$
  required in \ref{item:extension_ii} naturally transfers to a smallness
  assumption on the associated boundary data $g$.  In contrast, although
  condition \ref{item:extension_iii} is a direct consequence of
  \ref{item:extension_ii}, it does not require such a condition.  As we shall
  see in Subsection \ref{subsec:extension}, 
  we can find an extension $\wvel$ satisfying
  \ref{item:extension_iii} without imposing a smallness assumption on $g$.
\end{rem}

Below, we will show existence to problem
\eqref{sys:PDE_perturbed} in the following sense.  

\begin{defn}\label{def:weaksol_PDEmodified}
We call a couple $\np{\vvel,\Stens}$ a \emph{weak solution to \eqref{sys:PDE_perturbed}}
if
$
\np{\vvel,\Stens}\in\LHprime\times\XTprime
$ 
for all $0<T'< T$,
if $\restriction{\vvel}{\partial\Omega\times(0,T)}=0$
and if the identities
\begin{align}
\begin{split}
&\int_0^T\int_\Omega \bb{
-\vvel\cdot\pt\Phi
-\np{\vvel{+}\wvel}\otimes\np{\vvel{+}\wvel}:\grad\Phi
+\eta \Stens:\grad\Phi
+\mu\grad\vvel:\grad\Phi
}\dx\dt
\\
&\qquad\qquad=\int_0^T\int_\Omega \bb{f_0\cdot\Phi-f_1:\grad\Phi}\dx\dt
+\int_\Omega\vvel_0\cdot\Phi(\cdot,0)\dx,
\end{split}
\label{eq:weaksol_perturbed_v}
\\
\begin{split}
&\int_0^T\int_\Omega \bb{
-\Stens\cdot\pt\Psi
+\np{\vvel{+}\wvel}\cdot\grad\Stens:\Psi
+\big(\Stens\rottensor\np{\vvel{+}\wvel} 
      -\rottensor\np{\vvel{+}\wvel}\Stens\big):\Psi
\\
&\qquad+\partial\potential(\Stens):\Psi
+\gamma\grad\Stens:\grad\Psi
-\eta \straintensor\np{\vvel{+}\wvel}:\Psi
}\dx\dt
=\int_\Omega\Stens_0:\Psi(\cdot,0)\dx
\end{split}
\label{eq:weaksol_perturbed_S}
\end{align}
hold for all $\Phi\in\CRcisigma\np{\OmegaTzero}$ and $\Psi\in\CRcidev(\OmegaTzerocl)$.
\end{defn}

We shall provide a proof of the following existence result.

\begin{thm}\label{thm:existence_PDEmodified}
Let $\vvel_0$, $\Stens_0$ and $f$ be as in \eqref{el:data_perturbed},
let $\potential\in\CR{1,1}(\LRdev{2}(\Omega))$
satisfy \eqref{eq:propP},
and let $\wvel$ be as in Assumption \ref{asm:extension}.
Then there exists a weak solution 
$\np{\vvel,\Stens}$
to \eqref{sys:PDE_perturbed}
in the sense of Definition \ref{def:weaksol_PDEmodified}.
Additionally, this solution is weakly continuous in $\LR{2}(\Omega)$, 
that is, 
\begin{equation}\label{eq:WeakContinuity}
\vvel\in\CRweak(0,T;\LR{2}(\Omega)^3), 
\qquad \Stens\in\CRweak(0,T;\LR{2}(\Omega)^{3\times3}),
\end{equation}
with $\bp{\vvel(0),\Stens(0)}=\bp{\vvel_0,\Stens_0}$,
and it satisfies the energy inequalities
\begin{align}
&\begin{aligned}
&\half\norm{\vvel(t)}_2^2
+\mu\norm{\grad\vvel}_{\LR{2}(\Omegat)}^2
\\
&
\leq\half \norm{\vvel_0}_2^2
+\int_0^t\! \int_{\Omega}\bb{f_0\cdot\vvel
-f_1:\grad\vvel
+\wvel\otimes\np{\vvel{+}\wvel}:\grad\vvel
-\eta\Stens:\grad\vvel} \dx\dtau,
\end{aligned}
\label{est:EnergyInequality_v}
\\
&\begin{aligned}
&\half\norm{\Stens(t)}_2^2
+\gamma\norm{\grad\Stens}_{\LR{2}(\Omegat)}^2
+\int_0^t\int_\Omega\partial\potential(\Stens):\Stens\dx\dtau
\\
&\qquad
\leq\half\norm{\Stens_0}_2^2 
+\int_0^t\int_\Omega \bb{-\wvel\cdot\grad\Stens:\Stens
+\eta\grad\np{\vvel{+}\wvel}:\Stens}\dx\dtau
\end{aligned}
\label{est:EnergyInequality_S}
\end{align}
for all $t\in[0,T)$.
In particular, we conclude the total energy inequality
\begin{equation}
  \label{est:EnergyInequality}
\begin{aligned}
&\half \bp{\norm{\vvel(t)}_{\LR{2}(\Omega)}^2
+\norm{\Stens(t)}_{\LR{2}(\Omega)}^2}
+\mu\norm{\grad\vvel}_{\LR{2}(0,t;\LR{2}(\Omega))}^2
+\gamma\norm{\grad\Stens}_{\LR{2}(0,t;\LR{2}(\Omega))}^2
\\
&+\int_0^t\!\int_\Omega\partial\potential(\Stens):\Stens\dx\dtau 
\leq \half \bp{ \norm{\vvel_0}_{\LR{2}(\Omega)}^2
+ \norm{\Stens_0}_{\LR{2}(\Omega)}^2}
\\
&+\int_0^t\!\int_\Omega 
\bb{f_0\cdot\vvel
-f_1:\grad\vvel
+\wvel\otimes\np{\vvel{+}\wvel}:\grad\vvel
-\wvel\cdot\grad\Stens:\Stens
+\eta\grad\wvel:\Stens}
\dx\dtau
\end{aligned}
\end{equation}
for all $t\in[0,T)$.
\end{thm}

\subsubsection{Approximate solutions}
\label{subsec:ApproxSol}

First, we construct a sequence of
approximate solutions to \eqref{sys:PDE_perturbed}.
To this end, we introduce suitable basis functions.

\begin{lem}\label{lem:ONBL2sigma}
  There exists a sequence $\np{\varphi_k}\subset\CRcisigma(\Omega)$, which is
  an orthonormal basis of $\LRsigma{2}(\Omega)$, such that for all
  $\Phi\in\CRcisigma(\OmegaTzero)$ and all $\varepsilon>0$ there exist 
  $N\in \N$ and $\gamma_1,\dots,\gamma_N\in\CRc{1}([0,T))$ such that
\begin{equation}\label{est:ONBL2sigma}
\max_{t\in\nb{0,T}} \norml{\sum_{k=1}^N\gamma_k(t)\varphi_k-\Phi}_{\CR{2}(\Omega)}
+\max_{t\in\nb{0,T}} \norml{\sum_{k=1}^N\pt\gamma_k(t)\varphi_k-\pt\Phi}_{\CR{1}(\Omega)}
<\varepsilon.
\end{equation}
\end{lem}
\begin{proof}
See \cite[Lemma 2.3]{GaldiReviewOnIVP}.
\end{proof}

\begin{lem}\label{lem:ONBL2dev}
  There exists a sequence $\np{\psi_k}\subset\CRidev(\overline{\Omega})$,
  which is an orthonormal basis of $\LRdev{2}(\Omega)$, such that for all
  $\Psi\in\CRcidev(\OmegaTzerocl)$ and all $\varepsilon>0$ there exist 
  $N\in \N$ and 
  $\widetilde\gamma_1,\dots,\widetilde\gamma_N\in\CRc{1}([0,T))$ such that
\begin{equation}\label{est:ONBL2dev}
\max_{t\in\nb{0,T}} \norml{\sum_{k=1}^N\widetilde\gamma_k(t)\psi_k-\Psi}_{\CR{2}(\Omega)}
+\max_{t\in\nb{0,T}} \norml{\sum_{k=1}^N\pt\widetilde\gamma_k(t)\psi_k-\pt\Psi}_{\CR{1}(\Omega)}
<\varepsilon.
\end{equation}
\end{lem}
\begin{proof}
One can follow the proof of \cite[Lemma 2.3]{GaldiReviewOnIVP}.
\end{proof}

\begin{rem}\label{rem:basis}
Observe that $(\varphi_k)$ is a basis of $\HSRNsigma{1}(\Omega)$,
and $(\psi_k)$ is a basis of $\HSRdev{1}(\Omega)$.
To see this, consider $\varphi\in\CRcisigma(\Omega)$ 
and let $\Phi\in\CRcisigma(\OmegaTzero)$
such that $\Phi\np{\cdot,0}=\varphi$.
Let $\varepsilon>0$ and $\gamma_{1},\dots,\gamma_{N}$
as in Lemma \ref{lem:ONBL2sigma}.
Then
\[
\norml{\sum_{k=1}^{N}\gamma_k(0)\varphi_k-\varphi}_{\CR{2}(\Omega)}
\leq
\max_{t\in\nb{0,T}}\norml{\sum_{k=1}^{N}\gamma_k(t)\varphi_k-\Phi(\cdot,t)}_{\CR{2}(\Omega)}
<\varepsilon.
\]
Since $\CRcisigma(\Omega)$ is dense in $\HSRNsigma{1}(\Omega)$
and $\Omega$ is bounded, this shows the claim for $(\varphi_k)$.
Taking $\psi\in\CRidev(\overline{\Omega})$
and $\Psi\in\CRcidev\np{\OmegaTzerocl}$ with $\Psi\np{\cdot,0}=\psi$ instead,
we can use Lemma \ref{lem:ONBL2dev}
to conclude the statement for $(\psi_k)$ in the same way.
\end{rem}

With these bases at hand, we now construct a sequence of approximate solutions 
in the following way.
For $k\in\N$
we call $\np{\vvel,\Stens}$ an approximate solution 
(of order $k$)
if there exist $\alpha_r,\,\beta_r\in\CR{1}(0,T)\cap\CR{0}([0,T))$, $r=1,\dots,k$,
such that
\begin{equation}\label{eq:approxSol_decomp}
\vvel(x,t)=\vvel_k(x,t)=\sum_{r=1}^k\alpha_r(t)\varphi_r(x),\qquad
\Stens(x,t)=\Stens_k(x,t)=\sum_{r=1}^k\beta_r(t)\psi_r(x),
\end{equation}
and for all $\ell\in\set{1,\dots,k}$ the pair $\np{\vvel,\Stens}$ satisfies
\begin{align}
\begin{aligned}
\int_\Omega \bb{
\pt\vvel\cdot\varphi_\ell
-\np{\vvel{+}\wvel}\otimes\np{\vvel{+}\wvel}:\grad\varphi_\ell
&+\eta \Stens:\grad\varphi_\ell
+\mu\grad\vvel:\grad\varphi_\ell
}\dx
\\
&=\int_\Omega 
f_0\cdot
\varphi_\ell\dx
-\int_\Omega 
f_1:\grad\varphi_\ell\dx,
\end{aligned}\label{eq:approxSol_v}
\\
\begin{aligned}
\int_\Omega \bb{
\pt\Stens\cdot\psi_\ell
+\np{\vvel{+}\wvel}\cdot\grad\Stens:\psi_\ell
+\Stens\rottensor\np{\vvel{+}\wvel}:\psi_\ell
-\rottensor\np{\vvel{+}\wvel}\Stens:\psi_\ell
&\\
+\partial\potential(\Stens):\psi_\ell
+\gamma\grad\Stens:\grad\psi_\ell
-\eta\grad\np{\vvel{+}\wvel}:\psi_\ell
&}\dx
=0
\end{aligned}\label{eq:approxSol_S}
\end{align}
in $(0,T)$ and
\begin{equation}
\int_\Omega\vvel(0)\cdot\varphi_\ell\dx
=\int_\Omega\vvel_0\cdot\varphi_\ell\dx,
\qquad
\int_\Omega\Stens(0):\psi_\ell\dx
=\int_\Omega\Stens_0:\psi_\ell\dx.
\label{eq:approxSol_IV}
\end{equation}
The existence of approximate solutions is guaranteed by the following result.

\begin{lem}\label{lem:approxSol_existence}
For all $k\in\N$ there exists an approximate solution 
$\np{\vvel,\Stens}=\np{\vvel_k,\Stens_k}$,
which satisfies the partial energy-dissipation equalities
\begin{subequations}
 \label{eq:ApprSolEDBal}
\begin{align}
&\begin{aligned}
&\half\norm{\vvel(t)}_2^2
+\mu\norm{\grad\vvel}_{\LR{2}(\Omegat)}^2
\\
&
=\half \norm{\vvel(0)}_2^2
+\int_0^t\int_{\Omega}\bb{f_0\cdot\vvel-f_1:\grad\vvel
+\wvel\otimes\np{\vvel{+}\wvel}:\grad\vvel
-\eta\Stens:\grad\vvel}\dx\dtau,
\end{aligned}
\label{eq:approxSol_EnergyEquality_v}
\\
&\begin{aligned}
&\half\norm{\Stens(t)}_2^2
+\gamma\norm{\grad\Stens}_{\LR{2}(\Omegat)}^2
+\int_0^t\int_\Omega\partial\potential(\Stens):\Stens\dx\dtau
\\
&\qquad
=\half\norm{\Stens(0)}_2^2 
+\int_0^t\int_\Omega \bb{-\wvel\cdot\grad\Stens:\Stens
+\eta\grad\np{\vvel{+}\wvel}:\Stens}\dx\dtau
\end{aligned}
\label{eq:approxSol_EnergyEquality_S}
\end{align}
\end{subequations}
for all $t\in(0,T)$.
Moreover, for $0<T'< T$, there exists a constant $M_{T'}>0$, 
only depending on the data and $T'$
but independent of $k$ and $\potential$,
such that
\begin{equation}\label{est:approxSol_uniformbound}
\begin{aligned}
\sup_{t\in(0,T')}\!\! \Bp{  \norm{\vvel(t)}_{2}^2 + \norm{\Stens(t)}_{2}^2 }
+\int_0^{T'}\!\!\!  \Bp{\norm{\vvel(t)}_{1,2}^2 
&+\norm{\Stens(t)}_{1,2}^2 + \potential(\Stens)} \!\dt
 \leq M_{T'},
\end{aligned}
\end{equation}
so that the sequence $\np{\vvel_k,\Stens_k}_{k\in\N}$ is bounded in $\LHprime\times\XTprime$.
\end{lem}

\begin{rem}\label{rem:approxSol_totalEnergy}
  The energy balances of the kinetic energy and of the stored elastic energy
  are expressed in \eqref{eq:ApprSolEDBal} in two separate equations.
  By summation we obtain the total energy equality
\begin{equation}\label{eq:approxSol_EnergyEquality}
\begin{aligned}
&\half\norm{\vvel(t)}_2^2
+\half\norm{\Stens(t)}_2^2
+\mu\norm{\grad\vvel}_{\LR{2}(\Omegat)}^2
+\gamma\norm{\grad\Stens}_{\LR{2}(\Omegat)}^2
\\
&+\int_0^t\int_\Omega\partial\potential(\Stens):\Stens\dx\dtau
=\half \norm{\vvel(0)}_2^2+\half\norm{\Stens(0)}_2^2 
\\
&+\int_0^t\int_\Omega \bb{f_0\cdot\vvel-f_1:\grad\vvel
+\wvel\otimes\np{\vvel{+}\wvel}:\grad\vvel-\wvel\cdot\grad\Stens:\Stens
+\eta\grad\wvel:\Stens}\dx\dtau.
\end{aligned}
\end{equation}
\end{rem}

\begin{proof}
We reduce the equations \eqref{eq:approxSol_v}--\eqref{eq:approxSol_IV} to
an initial-value problem for the coefficient function 
$\np{\alpha,\beta}=\np{\alpha_1,\dots,\alpha_k,\beta_1,\dots,\beta_k}$. 
Let $\np{\cdot}^\prime=\ddt$ denote the time derivative. 
Due to orthogonality properties of the two bases $\np{\varphi_k}$ and $\np{\psi_k}$,
we then obtain
\begin{equation}\label{sys:coeffIVP}
\begin{aligned}
\alpha^\prime(t) &= F^1(\alpha(t),\beta(t),t), 
&
\beta^\prime(t) &=F^2(\alpha(t),\beta(t)), 
\\
\alpha_\ell(0) &= \int_\Omega \vvel_0\cdot\varphi_\ell\dx,
&
\beta_\ell(0) &= \int_\Omega\Stens : \psi_\ell\dx 
\qquad 
\np{\ell=1,\dots,k},
\end{aligned}
\end{equation}
where $F^j=(F^j_1,\dots,F^j_k)$, $j=1,2$, with
\[
\begin{aligned}
F^1_\ell(\alpha,\beta,t)
&=\sum_{r,s=1}^n\alpha_r\alpha_s 
\int_\Omega\varphi_r\otimes\varphi_s:\grad\varphi_\ell\dx
\\
&\quad
+\sum_{r=1}^k\alpha_r\int_\Omega \bp{\wvel\otimes\varphi_r+\varphi_r\otimes\wvel}:\grad\varphi_\ell\dx
+\int_\Omega \wvel\cdot\grad\wvel\cdot\varphi_\ell\dx
\\
&\quad
-\eta\sum_{r=1}^k\beta_r\int_\Omega\psi_r:\grad\varphi_\ell\dx
-\mu\sum_{r=1}^k\alpha_r\int\nabla\varphi_r:\nabla\varphi_\ell\dx
\\
&\quad
+\int_\Omega 
f_0(\cdot,t)\cdot\varphi_\ell
-f_1(\cdot,t):\grad\varphi_\ell\dx,
\end{aligned}
\]
and
\[
\begin{aligned}
&F^2_\ell(\alpha,\beta)
=-\sum_{r,s=1}^k\alpha_r\beta_s\int\varphi_r\cdot\grad\psi_s:\psi_\ell\dx
-\sum_{r=1}^k\beta_r\int\wvel\cdot\grad\psi_s:\psi_\ell\dx 
\\
&\qquad
-\sum_{r,s=1}^k\beta_r\alpha_s\int_\Omega
\bb{\psi_r\rottensor\np{\varphi_r}-\rottensor\np{\varphi_r}\psi_s}:\psi_\ell\dx
\\
&\qquad
-\int_\Omega\bb{\psi_r\rottensor\np{\wvel}-\rottensor\np{\wvel}\psi_s}:\psi_\ell\dx
\\
&\qquad
-\int_\Omega\partial\potential(\sum_{r=1}^k\alpha_r\psi_r):\psi_\ell\dx
-\gamma\sum_{r=1}^k\beta_r\int_\Omega\grad\psi_r:\grad\psi_\ell\dx 
\\
&\qquad
+\eta\sum_{r=1}^k\alpha_r\int_\Omega\grad\varphi_r:\psi_\ell\dx
+\eta\int_\Omega\grad\wvel:\psi_\ell\dx
\end{aligned}
\]
for $\ell=1,\dots,k$.
In particular, since $\partial\potential$ is Lipschitz continuous by assumption,
\eqref{sys:coeffIVP} 
is an initial-value problem with right-hand side $F=(F^1,F^2)$
that satisfies a local Lipschitz condition.
By the Picard--Lindel\"of theorem we thus obtain a
unique local solution $(\alpha,\beta)$ to \eqref{sys:coeffIVP}.
Let $(0,T_k)\subset(0,T)$ denote its maximal existence interval. 
Then $\np{\vvel,\Stens}$ defined by \eqref{eq:approxSol_decomp}
satisfy \eqref{eq:approxSol_v}--\eqref{eq:approxSol_S} in $(0,T_k)$
and \eqref{eq:approxSol_IV}.

To conclude the energy equalities \eqref{eq:approxSol_EnergyEquality_v} and
\eqref{eq:approxSol_EnergyEquality_S} for all $t\in(0,T_k)$, we multiply
\eqref{eq:approxSol_v} by $\alpha_\ell$ and \eqref{eq:approxSol_S} by
$\beta_\ell$, sum over $\ell=1,\dots,k$ and integrate over the time interval
$(0,t)$.  This leads to the two equalities \eqref{eq:ApprSolEDBal} by 
employing the identities
\[
\begin{aligned}
\int_\Omega\vvel\otimes\np{\vvel{+}\wvel}:\grad\vvel\dx
&=\half\int_\Omega\np{\vvel{+}\wvel}\cdot\grad\snorm{\vvel}^2\dx
=\half\int_{\partial\Omega}\np{\vvel{+}\wvel}\cdot\nvec\snorm{\vvel}^2\,\dsigma
=0, \\
\int_\Omega\vvel\cdot\grad\Stens:\Stens\dx
&=\half\int_\Omega\vvel\cdot\grad\snorm{\Stens}^2\dx
=\half\int_{\partial\Omega}\vvel\cdot\nvec\snorm{\Stens}^2\,\dsigma
=0,
\end{aligned}
\] 
which hold due to $\vvel=0$ on $\partial\OmegaTk$.

Next we show that $T_k=T$.
To this end, we add \eqref{eq:approxSol_EnergyEquality_v} and 
\eqref{eq:approxSol_EnergyEquality_S} to obtain \eqref{eq:approxSol_EnergyEquality}
and further estimate the right-hand side 
of \eqref{eq:approxSol_EnergyEquality} for $t\in(0,T_k)$. 
The initial terms can be estimated with
Bessel's inequality as
\begin{equation}\label{est:approxSol_IV}
\norm{\vvel(0)}_2^2\leq\norm{\vvel_0}_{2}^2,
\qquad
\norm{\Stens(0)}_2^2\leq\norm{\Stens_0}_{2}^2.
\end{equation}
For the linear terms,
we employ a combination of H\"older's and Young's inequalities to obtain
\begin{align}
\int_0^t\int_\Omega f_0\cdot\vvel\dx\dtau
&\leq \Cc{c}(\varepsilon) \norm{f_0}_{\LR{1}(0,t;\LR{2}(\Omega))}^2 
+ \varepsilon \norm{\vvel}_{\LR{\infty}(0,t;\LR{2}(\Omega))}^2,
\label{est:approxSol_f0v}
\\
\int_0^t\int_\Omega f_1:\grad\vvel\dx\dtau
&\leq \Cc{c}(\varepsilon) \norm{f_1}_{\LR{2}(\Omegat)}^2 
+ \varepsilon \norm{\grad\vvel}_{\LR{2}(\Omegat)}^2,
\label{est:approxSol_f1v}
\\
\int_0^t\int_\Omega \grad\wvel:\Stens\dx\dtau
&\leq \Cc{c}(\varepsilon) \norm{\grad\wvel}_{\LR{1}(0,t;\LR{2}(\Omega))}^2 
+ \varepsilon \norm{\Stens}_{\LR{\infty}(0,t;\LR{2}(\Omega))}^2,
\label{est:approxSol_dwS}
\\
\int_0^t\int_\Omega \np{\wvel\otimes\wvel}:\grad\vvel\dx\dtau
&\leq \Cc{c}(\varepsilon) \norm{\wvel}_{\LR{4}(\Omegat)}^4 
+ \varepsilon \norm{\grad\vvel}_{\LR{2}(\Omegat)}^2
\label{est:approxSol_wwdv}
\end{align}
for any $\varepsilon>0$.
Next we address the nonlinear terms, where we need to discuss the
different cases in Assumption \ref{asm:extension}. 

\underline{Case $s<\infty$:} We use part (a) of Assumption \ref{asm:extension}
with $r>3$ and define $p\in (2,6)$ via $1/p=1/2-1/r$. With
$\theta=3/2-3/p=3/r=1-2/s \in (0,1)$ the Gagliardo--Nirenberg inequality gives
\[
\begin{aligned}
\norm{\vvel(t)}_p
&\leq \Cc{c}\norm{\vvel(t)}_2^{1-\theta}\norm{\grad\vvel(t)}_2^{\theta},
\\
\norm{\Stens(t)}_p
&\leq \Cc{c}\norm{\Stens(t)}_2^{1-\theta}\norm{\grad\Stens(t)}_2^{\theta}
+\Cc{c}\norm{\Stens(t)}_2
\end{aligned}
\]
for all $t\in[0,T_k]$,
so that H\"older's and Young's inequalities lead to
\begin{align}
\begin{split}
&\int_0^t\int_\Omega
\wvel\otimes\vvel:\grad\vvel\dx\dtau
\leq\int_0^t \norm{\wvel}_{r}\norm{\vvel}_{p}\norm{\grad\vvel}_{2}\dtau
\\
&\qquad\leq\Cc{c}\int_0^t \norm{\wvel}_{r}\norm{\vvel}_{2}^{1-\theta}\norm{\grad\vvel}_{2}^{1+\theta}\dtau 
\\
&\qquad\leq \varepsilon \norm{\grad\vvel}_{\LR{2}\np{\Omegat}}^2
+\Cc{c}(\varepsilon)\int_0^t \norm{\wvel}_{r}^s\norm{\vvel}_{2}^2\dtau,
\end{split}
\label{est:approxSol_wvdv}
\\
\begin{split}
&\int_0^t\int_\Omega\wvel\cdot\grad\Stens:\Stens\dx\dtau
\leq\int_0^t\norm{\wvel}_{r}\norm{\grad\Stens}_{2}\norm{\Stens}_{p}\dtau 
\\
&\qquad\leq\Cc{c}\int_0^t\norm{\wvel}_{r}\Bp{\norm{\grad\Stens}_{2}^{1+\theta}\norm{\Stens}_{2}^{1-\theta}
+\norm{\grad\Stens}_{2}\norm{\Stens}_{2}}\dtau
\\
&\qquad\leq\varepsilon\norm{\grad\Stens}_{\LR{2}\np{\Omegat}}^2
+\Cc{c}(\varepsilon)\int_0^t\bp{\norm{\wvel}_{r}^s+\norm{\wvel}_{r}^2}\norm{\Stens}_{2}^2\dtau.
\end{split}
\label{est:approxSol_wdSS}
\end{align}
Let now $t'\in(0,T_k)$ be arbitrary. 
Using \eqref{est:approxSol_IV}--\eqref{est:approxSol_wdSS}
and choosing $\varepsilon=\frac{1}{6}$,
for $t\in[0,t']$ we may then estimate the right-hand side of~\eqref{eq:approxSol_EnergyEquality} as
\[
\begin{aligned}
	&\norm{\vvel(t)}_{\LR{2}(\Omega)}^2
	+\norm{\Stens(t)}_{\LR{2}(\Omega)}^2
	+\norm{\grad\vvel}_{\LR{2}(\Omegat)}^2
	+\norm{\grad\Stens}_{\LR{2}(\Omegat)}^2
	\\
	&\qquad\qquad\qquad\hspace{5cm}
	+\int_0^t\int_\Omega\partial\potential(\Stens):\Stens\,
	\dx\dtau
	\\
	&
	\leq\Cc{c}\Bp{
		\norm{\vvel_0}_2^2+\norm{\Stens_0}_2^2 
		+\norm{f_0}_{\LR{1}(0,t;\LR{2}(\Omega))}^2
		+\norm{f_1}_{\LR{2}(\Omegat)}^2 
		+\norm{\grad\wvel}_{\LR{1}(0,t;\LR{2}(\Omega))}^2
		\\
		&\qquad\qquad\qquad
		+\norm{\wvel}_{\LR{4}(\Omegat)}^4
		+\int_0^t \bp{
			\norm{\wvel}_{r}^s
			+\norm{\wvel}_{r}^2}
		\bp{\norm{\vvel}_{2}^2+\norm{\Stens}_{2}^2}\dtau
	}
\\&\; +\frac{1}{2} \big(
\norm{\vvel}_{\LR{\infty}(0,t';\LR{2}(\Omega))}^2
{+}\norm{\Stens}_{\LR{\infty}(0,t';\LR{2}(\Omega))}^2
{+}\norm{\grad\vvel}_{\LR{2}\np{\Omega\times(0,t')}}^2
{+}\norm{\grad\Stens}_{\LR{2}\np{\Omega\times(0,t')}}^2
\big).
\end{aligned}
\]
We now take the supremum over $t\in[0,t']$ and 
make use of the elementary inequality $\sum_{i=1}^M\sup_t N_i(t)\le M\sup_t(\sum_{i=1}^M N_i(t))$ for $M\in\mathbb{N}$ 
and functions $N_i\ge0$. 
After rearranging terms, relabelling $t'$ by $t$,
and adding the squared norm of $\vvel$ and $\Stens$ in $\LR{2}(\Omegat)$, this yields the inequality
\[
\begin{aligned}
		&\norm{\vvel}_{\LR{\infty}(0,t;\LR{2}(\Omega))}^2
	+\norm{\Stens}_{\LR{\infty}(0,t;\LR{2}(\Omega))}^2
	+\norm{\vvel}_{\LR{2}\np{0,t;\HSR{1}(\Omega)}}^2
	+\norm{\Stens}_{\LR{2}\np{0,t;\HSR{1}(\Omega)}}^2 
	\\
	&\qquad\qquad\qquad\hspace{5cm}
	+\int_0^t\int_\Omega\partial\potential(\Stens):\Stens\,
	\dx\dtau
	\\
	&
	\leq\Cc{c}\Bp{
		\norm{\vvel_0}_2^2+\norm{\Stens_0}_2^2 
		+\norm{f_0}_{\LR{1}(0,t;\LR{2}(\Omega))}^2
		+\norm{f_1}_{\LR{2}(\Omegat)}^2 
		+\norm{\grad\wvel}_{\LR{1}(0,t;\LR{2}(\Omega))}^2
		\\
		&\qquad\qquad\qquad
		+\norm{\wvel}_{\LR{4}(\Omegat)}^4
		+\int_0^t \bp{
			\norm{\wvel}_{r}^s
			+\norm{\wvel}_{r}^2+1}
		\bp{\norm{\vvel}_{2}^2
			+\norm{\Stens}_{2}^2}\dtau
	}.
\end{aligned}
\]
Since the last term in the last line is bounded above by 
\[\int_0^t \bp{
	\norm{\wvel}_{r}^s
	+\norm{\wvel}_{r}^2+1}
	\bp{\norm{\vvel}_{\LR{\infty}(0,\tau;\LR{2}(\Omega))}^2
	+\norm{\Stens}_{\LR{\infty}(0,\tau;\LR{2}(\Omega))}^2}\dtau,
\]
an application of Gronwall's inequality 
eventually leads to the bound
\begin{equation}\label{est:approxSol_Gronwall}
	\begin{aligned}
		&\norm{\vvel}_{\LR{\infty}(0,t;\LR{2}(\Omega))}^2
		+\norm{\Stens}_{\LR{\infty}(0,t;\LR{2}(\Omega))}^2
		+\norm{\vvel}_{\LR{2}\np{0,t;\HSR{1}(\Omega)}}^2
		+\norm{\Stens}_{\LR{2}\np{0,t;\HSR{1}(\Omega)}}^2 
		\\
		&\qquad+\int_0^t\potential(\Stens)\dtau
		\leq\Cc{c}\bp{
			\norm{\vvel_0}_2^2+\norm{\Stens_0}_2^2 
			+\norm{f_0}_{\LR{1}(0,t;\LR{2}(\Omega))}^2
			+\norm{f_1}_{\LR{2}(\Omegat)}^2 
			\\
			&\qquad
			+\norm{\grad\wvel}_{\LR{1}(0,t;\LR{2}(\Omega))}^2
			+\norm{\wvel}_{\LR{4}(\Omegat)}^4}
		\exp\Bp{\Cc{c}\int_0^t \bp{
				\norm{\wvel}_{r}^s
				+\norm{\wvel}_{r}^2+1}
			\dtau}
	\end{aligned}
\end{equation}
for $t\in(0,T_k)$ in the case $s<\infty$,
where we also used \eqref{est:dPS}.

\underline{Case $s=\infty$:} We can apply
\eqref{est:nonlinearterm_smallness_v} and \eqref{est:nonlinearterm_smallness_S}
to obtain
\begin{align}
\int_0^t\int_\Omega
\wvel\otimes\vvel:\grad\vvel\dx\dtau
&\leq \frac{\mu}{2} \norm{\grad\vvel}_{\LR{2}\np{\Omegat}}^2,
\label{est:approxSol_wvdv_sinfty}
\\
\int_0^t\int_\Omega\wvel\cdot\grad\Stens:\Stens\dx\dtau
&\leq\frac{\gamma}{2}\bp{
\norm{\Stens}_{\LR{2}\np{\Omegat}}^2
+\norm{\grad\Stens}_{\LR{2}\np{\Omegat}}^2}.
\label{est:approxSol_wdSS_sinfty}
\end{align}
Using \eqref{est:approxSol_wvdv_sinfty} and \eqref{est:approxSol_wdSS_sinfty}
instead of \eqref{est:approxSol_wvdv} and \eqref{est:approxSol_wdSS},
the argument from above leads to
\begin{equation}\label{est:approxSol_Gronwall_sinfty}
\begin{aligned}
&\norm{\vvel}_{\LR{\infty}(0,t;\LR{2}(\Omega))}^2
+\norm{\Stens}_{\LR{\infty}(0,t;\LR{2}(\Omega))}^2
+\norm{\vvel}_{\LR{2}\np{0,t;\HSR{1}(\Omega)}}^2
+\norm{\Stens}_{\LR{2}\np{0,t;\HSR{1}(\Omega)}}^2 \\
&
+\int_0^t\potential(\Stens)\dtau
\leq\Cc{c}\bp{
\norm{\vvel_0}_2^2+\norm{\Stens_0}_2^2 
+\norm{f_0}_{\LR{1}(0,t;\LR{2}(\Omega))}^2
+\norm{f_1}_{\LR{2}(\Omegat)}^2 
\\
&\qquad\qquad\qquad\qquad\qquad\qquad
+\norm{\grad\wvel}_{\LR{1}(0,t;\LR{2}(\Omega))}^2
+\norm{\wvel}_{\LR{4}(\Omegat)}^4}
\e^{\Cc{c}t}
\end{aligned}
\end{equation}
in the case $s=\infty$.

Now consider general $s\in(2,\infty]$ again.
For any $T'\in(0,T_k)$, 
the right-hand side of \eqref{est:approxSol_Gronwall} and \eqref{est:approxSol_Gronwall_sinfty}
can be bounded uniformly in $t\in(0,T')$ 
by a constant $M_{T'}>0$ that only depends on the data and $T'$. 
In particular, $M_{T'}$ is independent of $k\in\N$ and $\potential$,
and we conclude \eqref{est:approxSol_uniformbound} in both cases.
By Parseval's identity, 
this shows
\[
\sum_{r=1}^k \snorm{\alpha_r(t)}^2+\snorm{\beta_r(t)}^2 
= \norm{\vvel(t)}_2^2+\norm{\Stens(t)}_2^2\leq M_{T'}
\]
for all $t\in(0,T')$.
Hence the solution $\np{\alpha,\beta}$ does not blow-up at $t=T'$, 
and we conclude $T_k=T$
together with
\eqref{est:approxSol_uniformbound}.
\end{proof}

\begin{lem}\label{lem:approxSol_timederivative}
For any  $0<T'< T$,
the sequence $\np{\partial_t\vvel_k,\partial_t\Stens_k}_{k\in\N}$
is bounded in 
\[
\LR{1}\bp{0,T';\bp{\HSRNsigma{1}(\Omega)}'}\times\LR{8/7}\bp{0,T';\bp{\HSR{1}(\Omega)^{3\times3}}'}
\]
with
\[
\begin{aligned}
\norm{\pt\vvel_k}_{\LR{1}\np{0,T';\np{\HSRNsigma{1}(\Omega)}'}}
&\leq M_{T'}',
\\
\norm{\pt\Stens_k}_{\LR{8/7}\np{0,T';\np{\HSR{1}(\Omega)}'}}
&\leq M_{T',\potential},
\end{aligned}
\]
where $M_{T'}'$ is independent of $\potential$,
but $M_{T',\potential}$ depends on $\potential$.
\end{lem}

\begin{proof}
These bounds follow as for the classical Navier--Stokes equations,
and we only give a brief sketch here,
where we focus on the estimate of $(\pt\Stens_k)$.
The bound for $(\pt\vvel_k)$ follows similarly.
Let $\Stens=\Stens_k$ and recall the 
interpolation inequality
\[
\norm{\Stens(t)}_{4}
\leq\Cc{c}\norm{\Stens(t)}_{2}^{1/4}\norm{\Stens(t)}_{1,2}^{3/4}
\]
for a.a.~$t\in(0,T)$.
Since $\Stens$ is a finite linear combination of elements of the basis $(\psi_\ell)$
of $\HSRdev{1}(\Omega)$,
one deduces from \eqref{eq:approxSol_S}
that
$\pt\Stens(t)\in\bp{\HSR{1}(\Omega)}'$ for a.a.~$t\in(0,T)$
and 
\[
\begin{aligned}
&\norm{\pt\Stens}_{\np{\HSR{1}(\Omega)}'}
\leq \Cc{c}\bp{
\norm{\vvel}_2^{1/4}\norm{\vvel}_{2}^{3/4}\norm{\grad\Stens}_{2}
+\norm{\wvel}_4\norm{\grad\Stens}_{2}
\\
&\qquad\qquad\qquad\quad
+\norm{\Stens}_{2}^{1/4}\norm{\Stens}_{1,2}^{3/4}\norm{\grad\np{\vvel{+}\wvel}}_{2}
+\norm{\Stens}_2
+\norm{\grad\Stens}_{2}
+\norm{\grad\np{\vvel{+}\wvel}}_{2}
},
\end{aligned}
\]
where we used the estimate 
$\norm{\potential(\Stens)}_{2}\leq \Cc{c}\norm{\Stens}_{2}$
for some $\potential$-dependent constant $\Cclast{c}>0$,
which follows from the Lipschitz continuity of $\partial\potential$  
and $\partial\potential(0)=0$.
Using H\"older's inquality and \eqref{est:approxSol_uniformbound},
one concludes the asserted uniform bound for $(\pt\Stens_k)$.
\end{proof}

\subsubsection{Existence of weak solutions to \eqref{sys:PDE_perturbed}}
\label{subsec:ExistenceWeakSolModifiedSystem}

Based on the previous preparations
we establish the existence of weak solutions to \eqref{sys:PDE_perturbed}
as stated in Theorem \ref{thm:existence_PDEmodified}.

\begin{proof}[Proof of Theorem \ref{thm:existence_PDEmodified}]
Let $\np{\vvel_k,\Stens_k}_{k\in\N}$
be the sequence of approximate solutions in $(0,T)$
from Lemma 
\ref{lem:approxSol_existence}.
Due to the uniform bounds from \eqref{est:approxSol_uniformbound} 
for $T'\in(0,T)$,
combined with a classical diagonalization argument,
we obtain the existence of a subsequence, which we also denote by $\np{\vvel_k,\Stens_k}$,
and a pair $\np{\vvel,\Stens}$ with
$\np{\vvel,\Stens}\in\LHprime\times\XTprime$ for each $T'\in(0,T)$
such that
\[
\begin{aligned}
\vvel_k&\rightharpoonup\vvel && \tin \LR{2}(0,T';\HSR{1}(\Omega)^3),
\\
\Stens_k&\rightharpoonup\Stens && \tin \LR{2}(0,T';\HSR{1}(\Omega)^{3\times3}),
\\
\vvel_k&\overset{\ast}{\rightharpoonup}\vvel && \tin \LR{\infty}(0,T';\LRsigma{2}(\Omega)),
\\
\Stens_k&\overset{\ast}{\rightharpoonup}\Stens && \tin \LR{\infty}(0,T';\LRdev{2}(\Omega)),
\\
\end{aligned}
\]
as $k\to\infty$.
Moreover, since by Lemma \ref{lem:approxSol_timederivative} the sequences
$(\pt\vvel_k)$ and $(\pt\Stens_k)$ are bounded in $\LR{1}(0,T';\bp{\HSRNsigma{1}(\Omega)}')$ and 
$\LR{8/7}(0,T';\bp{\HSR{1}(\Omega)^{3\times3}}')$, respectively,
and since the Sobolev embedding $\HSR{1}(\Omega)\embeds\LR{2}(\Omega)$ is compact,
the Aubin--Lions lemma further implies the strong convergence
\[
\begin{aligned}
\vvel_k&\to\vvel && \tin \LR{2}(0,T';\LR{2}(\Omega)^3),
\\
\Stens_k&\to\Stens && \tin \LR{2}(0,T';\LR{2}(\Omega)^{3\times3})
\end{aligned}
\]
as $k\to\infty$.
Let us show that 
$\np{\vvel,\Stens}$ is a weak solution to \eqref{sys:PDE_perturbed},
that is, that \eqref{eq:weaksol_perturbed_v} and \eqref{eq:weaksol_perturbed_S}
are satisfied.
To this end, let $\chi\in\CRci([0,T))$
and let $T'\in(0,T)$ such that $\supp\chi\subset[0,T')$.
Fix $\ell\in\N$.
Multiply \eqref{eq:approxSol_v}
and \eqref{eq:approxSol_S}
by $\chi(t)$, integrate over $(0,T)$
and pass to the limit $k\to\infty$
exploiting the above convergence properties.
For example, 
in view of \eqref{eq:approxSol_IV},
for $k\geq\ell$ we have 
\[
\begin{aligned}
\int_0^T\int_\Omega\pt\vvel_k\cdot\varphi_\ell\chi\dx\dt
&=-\int_0^T\int_\Omega\vvel_k\cdot\varphi_\ell\pt\chi\dx\dt
-\int_\Omega\vvel_0\cdot\varphi_\ell\chi(0)\dx
\\
&\to-\int_0^T\int_\Omega\vvel\cdot\varphi_\ell\pt\chi\dx\dt
-\int_\Omega\vvel_0\cdot\varphi_\ell\chi(0)\dx.
\end{aligned}
\]
Employing the strong convergence of $(\vvel_k)$ in $\LR{2}(\OmegaTprime)$,
we further conclude
\[
\int_0^T\int_\Omega\vvel_k\otimes\vvel_k:\grad\varphi_\ell\chi\dx\dt
\to\int_0^T\int_\Omega\vvel\otimes\vvel:\grad\varphi_\ell\chi\dx\dt.
\]
Similarly, we can derive
\[
\int_0^T\int_\Omega\partial\potential(\Stens_k):\psi_\ell\chi\dx\dt
\to\int_0^T\int_\Omega\partial\potential(\Stens):\psi_\ell\chi\dx\dt
\]
from the Lipschitz continuity of $\partial\potential$ 
and the strong convergence of $(\Stens_k)$ in $\LR{2}(\OmegaTprime)$.
Convergence of the remaining terms can be shown in a similar fashion,
and we conclude
\eqref{eq:weaksol_perturbed_v} and \eqref{eq:weaksol_perturbed_S}
for all $\Phi$ and $\Psi$ of the form
$\Phi(x,t)=\varphi_\ell(x)\chi(t)$ and $\Psi(x,t)=\psi_\ell(x)\chi(t)$
with $\ell\in\N$.
Finally, an approximation argument based on 
Lemma \ref{lem:ONBL2sigma}
and Lemma \ref{lem:ONBL2dev}
allows us to pass to general $\Phi\in\CRcisigma\np{\OmegaTzero}$
and $\Psi\in\CRcidev\np{\OmegaTzerocl}$, respectively.
Consequently, $\np{\vvel,\Stens}$ is a weak solution to \eqref{sys:PDE_perturbed}.

Now let us show the energy inequalities \eqref{est:EnergyInequality_v} 
and \eqref{est:EnergyInequality_S}.
Similarly to \cite[Proof of Theorem 3.1]{GaldiReviewOnIVP},
one can show that $\np{\vvel_k(t),\Stens_k(t)}\rightharpoonup\np{\vvel(t),\Stens(t)}$ 
in $\LR{2}(\Omega)$
as $k\to\infty$ for all $t\in(0,T)$
after possibly modifying the solution $\np{\vvel,\Stens}$ 
on a set of measure zero in $(0,T)$.
This property and the weak convergence in $\LR{2}(0,T';\HSR{1}(\Omega))$ imply
\begin{align}
\half\norm{\vvel(t)}_2^2
+\mu\norm{\grad\vvel}_{\LR{2}(\Omegat)}^2
&\leq
\liminf_{k\to\infty}\bp{
\half\norm{\vvel_k(t)}_2^2
+\mu\norm{\grad\vvel_k}_{\LR{2}(\Omegat)}^2
},
\label{est:EnergyInequality_lhs_conv_v}
\\
\half\norm{\Stens(t)}_2^2
+\gamma\norm{\grad\Stens}_{\LR{2}(\Omegat)}^2
&\leq
\liminf_{k\to\infty}\bp{
\half\norm{\Stens_k(t)}_2^2
+\gamma\norm{\grad\Stens_k}_{\LR{2}(\Omegat)}^2
}.
\label{est:EnergyInequality_lhs_conv_S}
\end{align}
Moreover, the strong convergence $\Stens_k\to\Stens$ in $\LR{2}(\OmegaTprime)$
and the Lipschitz continuity of $\partial\potential$ lead to
\begin{equation}\label{eq:EnergyInequality_potterm_conv}
\lim_{k\to\infty}\int_0^t\int_\Omega\partial\potential(\Stens_k):\Stens_k\dx\dt
=\int_0^t\int_\Omega\partial\potential(\Stens):\Stens\dx\dt.
\end{equation}
By construction we further have 
$\norm{\vvel(0)}_{2}\leq\norm{\vvel_0}_{2}$ and
$\norm{\Stens(0)}_{2}\leq\norm{\Stens_0}_{2}$
due to Bessel's inequality,
and we can directly conclude
\begin{align}
\lim_{k\to\infty}
\int_0^t\int_\Omega 
\bb{f_0\cdot\vvel_k
-f_1:\grad\vvel_k
}
\dx\dt
&=\int_0^t\int_\Omega 
\bb{f_0\cdot\vvel
-f_1:\grad\vvel
}
\dx\dt,
\label{eq:EnergyInequality_rhslin_conv_v1}
\\
\lim_{k\to\infty}
\int_0^t\int_\Omega 
\wvel\otimes\wvel:\grad\vvel_k
\dx\dt
&=\int_0^t\int_\Omega 
\wvel\otimes\wvel:\grad\vvel
\dx\dt,
\label{eq:EnergyInequality_rhslin_conv_v2}
\\
\lim_{k\to\infty}
\int_0^t\int_\Omega 
\eta\grad\wvel:\Stens_k
\dx\dt
&=\int_0^t\int_\Omega 
\eta\grad\wvel:\Stens
\dx\dt
\label{eq:EnergyInequality_rhslin_conv_S}
\end{align}
since $f,\wvel\otimes\wvel,\grad\wvel\in\LR{2}\np{\OmegaTprime}$.
Moreover, the strong convergence of $(\Stens_k)$ in $\LR{2}(\OmegaTprime)$ 
and the weak convergence of $(\grad\vvel_k)$ in $\LR{2}(\OmegaTprime)$ 
imply
\begin{equation}\label{eq:EnergyInequality_rhsnonlinCouple_conv}
\lim_{k\to\infty}\int_0^t\int_\Omega\Stens_k:\grad\vvel_k\dx\dt
=\int_0^t\int_\Omega\Stens:\grad\vvel\dx\dt.
\end{equation}
For the remaining terms, assume for the moment that $\wvel\in\CRci(\OmegaT)$.
Then $(\wvel\otimes\vvel_k)$ converges to $\wvel\otimes\vvel$ 
strongly in $\LR{2}\np{\OmegaT}$.
Hence we obtain
\begin{equation}\label{eq:EnergyInequality_rhsnonlin1_conv}
\lim_{k\to\infty}\int_0^t\int_\Omega\wvel\otimes\vvel_k:\grad\vvel_k\dx\dt
=\int_0^t\int_\Omega\wvel\otimes\vvel:\grad\vvel\dx\dt.
\end{equation}
For general $\wvel\in\LR{s}\np{0,T;\LR{r}(\Omega)}$
we obtain \eqref{eq:EnergyInequality_rhsnonlin1_conv} by approximating 
$\wvel$ by elements from $\CRci(\OmegaT)$
and exploiting that $\LH\embeds\LR{q}\np{0,T;\LR{p}(\Omega)}$ 
with $1/p=1/2-1/r$, $1/q=1/2-1/s$,
so that $2/q+3/p=3/2$.
Observe that here we use $\wvel\in\CR{0}(0,T;\LR{3}(\Omega))$ if $s=\infty$.
An analogous argument shows  
\begin{equation}\label{eq:EnergyInequality_rhsnonlin2_conv}
\lim_{k\to\infty}\int_0^t\int_\Omega\wvel\cdot\grad\Stens_k:\Stens_k\dx\dt
=\int_0^t\int_\Omega\wvel\cdot\grad\Stens:\Stens\dx\dt.
\end{equation}
Finally, we combine 
\eqref{est:EnergyInequality_lhs_conv_v}--\eqref{eq:EnergyInequality_rhsnonlin2_conv}
with the energy equalities \eqref{eq:approxSol_EnergyEquality_v} and 
\eqref{eq:approxSol_EnergyEquality_S}
to conclude the energy inequalities \eqref{est:EnergyInequality_v}
and \eqref{est:EnergyInequality_S}.
Moreover, in the same way as for the
the classical Navier--Stokes initial-value problem
(see \cite[Lemma 2.2]{GaldiReviewOnIVP} for example),
the weak solution can be redefined on a set of measure zero
such that it is weakly continuous in the sense of \eqref{eq:WeakContinuity}.
This finishes the proof of Theorem \ref{thm:existence_PDEmodified}. 
\end{proof}

\begin{rem}\label{rem:weaksol_PDEmodified_bounds}
From the proof of Theorem \ref{thm:existence_PDEmodified}
we directly obtain 
\[
\begin{aligned}
\norm{\vvel}_{\LR{\infty}\np{0,T';\LR{2}(\Omega)}}^2
+\norm{\Stens}_{\LR{\infty}\np{0,T';\LR{2}(\Omega)}}^2 
\qquad\qquad\qquad\qquad\qquad\qquad&\\
+\norm{\vvel}_{\LR{2}\np{0,T';\HSR{1}(\Omega)}}^2
+\norm{\Stens}_{\LR{2}\np{0,T';\HSR{1}(\Omega)}}^2
+\int_0^{T'}\!\potential(\Stens)\dt
&\leq M_{T'},
\\
\norm{\pt\vvel}_{\LR{1}\np{0,T';\np{\HSRNsigma{1}(\Omega)}'}}
&\leq M_{T'}',
\\
\norm{\pt\Stens}_{\LR{8/7}\np{0,T';\np{\HSR{1}(\Omega)}'}}
&\leq M_{T',\potential},
\end{aligned}
\]
for each $0<T'<T$,
where $M_{T'}$, $M_{T'}'$ and $M_{T',\potential}$ are given in 
Lemma \ref{lem:approxSol_existence} and Lemma \ref{lem:approxSol_timederivative}. 
\end{rem}

\subsection{Existence of a suitable extension}
\label{subsec:extension}

Since we have shown existence of a weak solution to \eqref{sys:PDE_perturbed}
in Theorem \ref{thm:existence_PDEmodified},
it remains to address the existence of solutions $\wvel$ 
to the Stokes initial-value problem \eqref{sys:StokesIVP}.
As mentioned above,
the forcing term $\widetilde{F}$ and the initial value $\wvel_0$ 
in \eqref{sys:StokesIVP} are not prescribed 
by the original problem,
whence we have some freedom in their choice.
For example, 
we can simply consider data $\widetilde{F}=\wvel_0=0$
and use existing theory for the Stokes initial-value problem
with inhomogeneous Dirichlet data
(see \cite{Grubb_NonhDirichletNSlowRegLpSobolevSpaces_2001,
FursikovGunzburgerHou_TraceTheorems3DTimeDependentVectorFields_2002,
Raymond_StokesNSNonhomBdryCond_2007} for example)
to obtain a suitable extension $\wvel$ 
satisfying Assumption \ref{asm:extension} in case \ref{item:extension_i}.
Proceeding like this for the cases \ref{item:extension_ii} and \ref{item:extension_iii} 
would require smallness of $g$.
In the following we focus on case \ref{item:extension_iii},
which is more general than \ref{item:extension_ii},
and show that smallness of $g$ is not necessary
if we exploit the freedom we have in the choice of $\widetilde{F}$
and $\wvel_0$.
For this purpose, we use the following 
well-known lemma.

\begin{lem}\label{lem:ext.nonl.small.v}
Let $\Omega$ be a bounded domain with connected $\CR{1,1}$-boundary,
$T\in(0,\infty]$
and let $g$
satisfy
\begin{subequations}\label{eq:Cond.g}
\begin{align}
\label{el:cond_g_lem}
&g\in\LR{\infty}\np{0,T;\HSR{1/2}(\partial\Omega)^3},\quad
\pt g\in\LR{\infty}\np{0,T;\HSR{-1/2}(\partial\Omega)^3},
\\
\label{eq:compatibility_g}
&\qquad\qquad\int_{\partial\Omega} g(t)\cdot\nvec=0\quad \text{for a.a.~}t\in(0,T).
\end{align}
\end{subequations}
Then, for each $\alpha>0$ there exists a function
\begin{equation}\label{el:extension_nonlinearterm_small_v}
\wvel_\alpha\in\LR{\infty}\np{0,T;\HSR{1}(\Omega)^3},\quad
\pt \wvel_\alpha\in\LR{\infty}\np{0,T;\HSR{-1}(\Omega)^3}
\end{equation}
with $\wvel_\alpha=g$ on $\partial\OmegaT$
and $\Div\wvel_\alpha=0$ in $\OmegaT$ such that
\[
\forall\, \vvel_1,\vvel_2\in\HSRN{1}(\Omega):\quad
\snormL{\int_\Omega 
\wvel_\alpha(t)\otimes\vvel_1:\grad\vvel_2
\dx}
\leq\alpha
\norm{\grad\vvel_1}_{2}\norm{\grad\vvel_2}_{2}
\]
and for all $q\in[1,\infty]$ there exists a constant $\Cc{C}=\Cclast{C}\np{\Omega,\alpha,q}>0$ 
such that
for all $T'\in(0,T)$ it holds
\begin{equation}\label{est:extension}
\begin{aligned}
\norm{\wvel_\alpha}_{\LR{q}(0,T';\HSR{1}(\Omega))}
&\leq \Cclast{C}\norm{g}_{\LR{q}(0,T';\HSR{1/2}(\partial\Omega))}
\\
\norm{\pt\wvel_\alpha}_{\LR{q}\np{0,T';\HSR{-1}(\Omega)}}
&\leq \Cclast{C}\norm{\pt g}_{\LR{q}\np{0,T';\HSR{-1/2}(\partial\Omega)}}.
\end{aligned}
\end{equation} 
\end{lem}

\begin{proof}
In 
\cite[Proposition 5.4]{FarwigKozonoSohr_GlobalLHWeakSolNSTimeDependentBdryData}
the statement was shown with $\alpha=1/4$.
An adaptation of the proof for arbitrary $\alpha>0$ is straightforward.
\end{proof}

In order to ensure \eqref{est:nonlinearterm_smallness_S},
it is sufficient to assume smallness of $g\cdot\nvec$ in a suitable norm.

\begin{lem}\label{lem:extension_nonlinearterm_small_S}
In the situation of Lemma \ref{lem:ext.nonl.small.v}
it holds
\[
\forall\, \Stens\in\HSR{1}(\Omega)^{3\times3}:\quad
\snormL{\int_\Omega 
\wvel_\alpha(t)\cdot\grad\Stens:\Stens
\dx}
\leq \Cc{C}\norm{g\cdot\nvec}_{\LR{\infty}(0,T;\LR{2}(\partial\Omega))}
\norm{\Stens}_{1,2}^2
\]
for some constant $\Cclast{C}=\Cclast{C}(\Omega)>0$.
\end{lem}

\begin{proof}
We have
\[
\int_\Omega 
\wvel_\alpha(t)\cdot\grad\Stens:\Stens
\dx
=\half\int_\Omega 
\wvel_\alpha(t)\cdot\grad\snorm{\Stens}^2
\dx
=\half\int_{\partial\Omega} 
g(t)\cdot\nvec\snorm{\Stens}^2
\,\dsigma.
\]
Now H\"older's inequality and Sobolev embeddings imply
\[
\snormL{
\int_\Omega 
\wvel_\alpha(t)\cdot\grad\Stens:\Stens
\dx}
\leq\Cc{c}\norm{g(t)\cdot\nvec}_{2}\norm{\Stens}_{\LR{4}(\partial\Omega)}^2
\leq\Cc{c}\norm{g(t)\cdot\nvec}_{2}\norm{\Stens}_{\HSR{1/2}(\partial\Omega)}^2.
\]
The statement now follows from a standard trace inequality.
\end{proof}

In particular, Lemma \ref{lem:extension_nonlinearterm_small_S}
implies that condition \eqref{est:nonlinearterm_smallness_S}
is satisfied automatically when $g\cdot\nvec=0$,
which is the case we are interested in.

\subsection{Existence of a weak solution}
\label{subsec:existence.weaksol}

Combining now Theorem \ref{thm:existence_PDEmodified},
Lemma \ref{lem:ext.nonl.small.v} and Lemma \ref{lem:extension_nonlinearterm_small_S},
we show existence of a weak solution to the original problem 
\eqref{sys:PDE}--\eqref{eq:InitialConditions}
in the sense of Definition \ref{def:weaksol}.

\begin{proof}[Proof of Theorem \ref{thm:existence_weaksol}]
  Let $\wvel=\wvel_\alpha$ from Lemma \ref{lem:ext.nonl.small.v} with
  $\alpha=\mu/2$.  Then the Aubin--Lions lemma implies
  $\wvel\in\CR{0}(0,T;\LR{2}(\Omega))$, and in virtue of $g\cdot\nvec=0$ and
  Lemma \ref{lem:extension_nonlinearterm_small_S}, we see that Assumption
  \ref{asm:extension} \ref{item:extension_iii} is satisfied.  From H\"older's
  inequality and Sobolev embeddings we further conclude the remaining
  properties of Assumption \ref{asm:extension}.  Since
  $\wvel\in\CR{0}(0,T;\LR{2}(\Omega))$, we can define
  $\wvel_0\coloneqq\wvel(\cdot,0)\in\LR{2}(\Omega)$.  Moreover, we set
  $\widetilde{F}\coloneqq\pt\wvel-\mu\Delta\wvel$.  Since every element of
  $\HSR{-1}(\Omega)^n$ can be represented as the divergence of a tensor field
  from $\LR{2}(\Omega)^{n\times n}$ (see \cite[Lemma
  1.6.1]{Sohr_TheNavierStokesEquations_2001} for example), we obtain
  $\widetilde{F}=\Div\widetilde{F}_1$ for some
  $\widetilde{F}_1\in\LR{\infty}(0,T;\LR{2}(\Omega)^{n\times n})$.  
  Now we set $f_0\coloneqq F_0$, $f_1\coloneqq F_1-\widetilde{F}_1$,
  $\vvel_0\coloneqq\Vvel_0-\wvel_0$, and let $\np{\vvel,\Stens}$ be the weak
  solution to \eqref{sys:PDE_perturbed} from Theorem
  \ref{thm:existence_PDEmodified}.  Since for all $\Phi\in\CRci(\OmegaTzero)^3$
  we have
\[
  \int_0^T\!\!\int_\Omega\! \bb{ -\wvel\cdot\pt\Phi +\mu\grad\wvel:\grad\Phi }
  \dx\dt =-\int_0^T\!\!\int_\Omega\! \widetilde{F}_1:\grad\Phi\dx\dt
  +\int_\Omega\!\wvel_0\cdot\Phi(\cdot,0)\dx,
\]
the pair $\np{\Vvel,\Stens}\coloneqq\np{\vvel{+}\wvel,\Stens}$ is a weak
solution to \eqref{sys:PDE}--\eqref{eq:InitialConditions} in the sense of
Definition \ref{def:weaksol}.  
Moreover, \eqref{eq:WeakContinuity_VS} follows
from \eqref{eq:WeakContinuity} and $\wvel\in\CR{0}(0,T;\LR{2}(\Omega))$,
and \eqref{est:EnergyInequality_v_combined} 
is a direct consequence of \eqref{est:EnergyInequality_v}.
Similarly to the proof of Lemma \ref{lem:extension_nonlinearterm_small_S},
we further derive
\[
\int_\Omega 
\wvel(t)\cdot\grad\Stens:\Stens
\dx
=\half\int_{\partial\Omega} 
g(t)\cdot\nvec\snorm{\Stens}^2
\,\dsigma=0
\]
since $g\cdot\nvec=0$. 
With this identity, \eqref{est:EnergyInequality_S_combined} directly follows from
\eqref{est:EnergyInequality_S}.
\end{proof}

We can further derive the following estimate,
which will be needed for passing from smooth Moreau
envelopes $\potential_\ve$ to nonsmooth
potentials $\potential$.

\begin{lem}
For all $0<T'<T$ 
there exists a constant $M_{T'}>0$, which is independent of $\potential$,
such that
\begin{equation}\label{est:weaksol_bounds}
\begin{aligned}
&\norm{\Vvel}_{\LR{\infty}\np{0,T';\LR{2}(\Omega)}}
+\norm{\Vvel}_{\LR{2}\np{0,T';\HSR{1}(\Omega)}}
+\norm{\pt\Vvel}_{\LR{1}\np{0,T';\np{\HSRNsigma{1}(\Omega)}'}}
\\
&\qquad+\norm{\Stens}_{\LR{\infty}\np{0,T';\LR{2}(\Omega)}}
+\norm{\Stens}_{\LR{2}\np{0,T';\HSR{1}(\Omega)}}
+\int_0^{T'}\potential(\Stens)\dt
\leq M_{T'}.
\end{aligned}
\end{equation}
\end{lem}

\begin{proof}
By \eqref{est:extension} we have
\[
\begin{aligned}
\norm{\wvel}_{\LR{\infty}\np{0,T';\LR{2}(\Omega)}}
&+\norm{\wvel}_{\LR{2}\np{0,T';\HSR{1}(\Omega)}}
+\norm{\pt\wvel}_{\LR{1}\np{0,T';\np{\HSRNsigma{1}(\Omega)}'}}
\\
&\leq\Cc{c}(T')\bp{
\norm{\wvel}_{\LR{\infty}\np{0,T';\HSR{1}(\Omega)}}
+\norm{\pt\wvel}_{\LR{\infty}\np{0,T';\HSR{-1}(\Omega)}}
}\\
&\leq\Cc{c}(T')\bp{
\norm{g}_{\LR{\infty}\np{0,T;\HSR{1/2}(\partial\Omega)}}
+\norm{\pt g}_{\LR{\infty}\np{0,T;\HSR{-1/2}(\partial\Omega)}}
}.
\end{aligned}
\]
In view of Remark \ref{rem:weaksol_PDEmodified_bounds} and the identity $\Vvel=\vvel+\wvel$, 
this shows \eqref{est:weaksol_bounds}.
\end{proof}

\section{Generalized solutions for nonsmooth potentials}
\label{su:GenSolNonsmooth}

The primary purpose of this section is to prove our main result,
Theorem~\ref{thm:genP}, on the existence of generalized solutions in case of a
nonsmooth potential $\mathcal{P}$, see Section~\ref{ssec:proof.main}. Recall
that throughout this manuscript $\potential\colon\LRdev{2}(\Om)\to[0,\infty]$
is assumed to be convex and lower semicontinuous with $\potential(0)=0.$ As
outlined in Section~\ref{ssec:strategy}, we proceed by regularization,
approximating $\mathcal{P}$ by its Moreau envelope $\{\mathcal{P}_\ve\}$, which
allows us to take advantage of Theorem~\ref{thm:existence_weaksol} providing
existence for smooth potentials. The key to passing to the limit $\ve\to0$ in
the tensorial evolutionary variational inequality for the approximate solutions
is the convergence result in Lemma~\ref{le:SW(V)}.

\subsection{Weak versus generalized solutions}
\label{ssec:weak.vs.gen}

Here, we provide a basic consistency analysis concerning our new notion of
solution for rate-type viscoelastoplastic models involving a
nonsmooth potential in the tensorial evolution. In particular, we will show
that the weak solutions from Theorem~\ref{thm:existence_weaksol} obtained in
the case of a smooth potential satisfy a variational inequality (for $S$) and
hence are generalized solutions in the sense of Definition~\ref{def:varsol}.
This result will further serve as a basic ingredient in the proof of
Theorem~\ref{thm:genP}.
%
We will also provide sufficient regularity conditions for $(V,S)$ that allow us
to conclude that generalized solutions are already weak solutions in the sense
of Definition \ref{def:weaksol}.  For this purpose, we need an approximation
property for the induced potential $\mathscr{P}$ acting on Bochner functions
\begin{align}\label{eq:def.scrP}
	\mathscr{P}(\tilde
	S):=\int_0^{T'}\mathcal{P}(\tilde S(t))\dt,\qquad \tilde S\in\LR{2}(0,T';\LR{2}_\delta(\Omega)).
\end{align}
The approximation condition states:
\begin{align}
	\begin{aligned}
	\label{eq:calP.appr.prop}
	\forall\, \tilde S\in \LR{2}(0,T';\LR{2}_\delta(\Omega)) &\ \exists\, (\tilde S_n)_{n\in \N} \subset
	\TestFn :
	\\&\tilde S_n\rightharpoonup  \tilde S\;\; \tin\LR{2}(0,T';\LR{2}(\Omega))\text{ and }  \mathscr P(\tilde S_n) \to \mathscr P(\tilde S). 
\end{aligned}
\end{align}
This property is certainly satisfied 
if $\mathscr P(S)=\int_0^{T'}\!\!\!\int_\Omega \mathfrak P(\Stens(x,t))\dx\dt$ for 
a continuous convex function $\mathfrak{P}\colon\R^{3\times3}_\delta\to[0,\infty)$ with $\mathfrak P(0)=0$ and a quadratic upper bound.
In this case 
it suffices to choose an arbitrary sequence $(\tilde \Stens_n)$ converging to $\tilde \Stens$  strongly in $\LR{2}(0,T';\LR{2}(\Omega))$, because $\mathscr P$ is norm continuous.
%
If $\mathfrak P$ is the sum of such a function 
and an \textit{indicator function} 
\begin{equation*}
\iota_A(S):= \left\{\begin{array}{cl} 0 & \text{for } 
	S\in A, \\ 
		\infty & \text{for } S\not\in A
	\end{array}  \right.
\end{equation*}
for a closed and convex subset $A\subset \R^{3\times3}_\delta$,
a suitable sequence $(\tilde \Stens_n)$ can be constructed by means of mollification
 taking into account the convexity and closedness of the set $A$ as well as the fact that $\mathfrak P(0)=0$ (which implies $0\in A$).
In particular, \eqref{eq:calP.appr.prop} holds for the plasticity potential $\mathcal P$
defined in \eqref{eq:pot.example.intro}. 

\begin{lem}[Weak versus generalized solutions]\label{l:vi.aux}
 \mbox{}\\
  (A)   Assume that $\mathcal P\in \CR{1,1}(\LR{2}_\delta(\Omega))$. 
  If $(V,S)$ is a weak solution in the sense of Definition~\ref{def:weaksol}
  that additionally satisfies the partial energy dissipation inequality~\eqref{est:EnergyInequality_S_combined}, 
  then it is also a generalized solution in the sense of
  Definition \ref{def:varsol}. 
\\[0.2em] 
(B) If $(V,S)$ is a generalized solution in the sense of
  Definition \ref{def:varsol} with the additional regularity
 \begin{equation}
    \label{eq:Cond.Gen.Weak}
  S\in \HSR{1}(0,T';\LR{2}(\Omega)) \cap \LR{2}(0,T';\HSR{2}(\Omega))\cap 
  \LR{\infty}(0,T';\LR{\infty}(\Omega)) 
  \end{equation}
 for all $T'<T$,
and if $\mathcal P$ satisfies the approximation property \eqref{eq:calP.appr.prop},
then it is also a weak solution in the sense of Definition
\ref{def:weaksol}, where \eqref{eq:weaksol_S} is replaced by 
\begin{equation}
\label{eq:weak.nonsmo}
\begin{split}
\int_0^T\int_\Omega \bb{
&{-}\Stens:\pt\Psi
+\Vvel\cdot\grad\Stens:\Psi
+\big(\Stens\rottensor\np{\Vvel}-\rottensor\np{\Vvel}\Stens\big):\Psi
\\
&+\beta:\Psi
+\gamma\grad\Stens:\grad\Psi
-\eta\straintensor(\Vvel):\Psi
}\dx\dt
=\int_\Omega\Stens_0:\Psi(0)\dx
\end{split}
\end{equation}
for all $\Psi\in\CRcidev(\OmegaTzerocl)$, where
$\beta \in \LR{2}(0,T;\LR{2}_\delta(\Omega))$ with
$\beta(t) \in \partial\mathcal P(S(t))$ for a.a.\ $t\in (0,T)$. 
\end{lem}

\begin{rem}
	By Proposition~\ref{prop:edin.genP}, generalized solutions with the extra regularity assumed in Lemma~\ref{l:vi.aux} (B) not only satisfy the weak formulation~\eqref{eq:weak.nonsmo} as asserted in the above lemma, but further fulfill the partial  energy inequality~\eqref{eq:edin.S.genP}, which is genuinely encoded in the evolutionary variational inequality.
\end{rem}

\begin{proof} \underline{Part (A):} Let $(V,S)$ be a weak solution. We
  have to derive the evolutionary variational inequality \eqref{eq:varin}.
    For this purpose,  we first show that the weak equation \eqref{eq:weaksol_S} for $S$
    implies a modified version
    that holds true for test functions $\tilde S\in \CRidev\np{\overline{\Omega}\times[0,T']}$, $T'<T$, not necessarily compactly supported in time. More precisely, we assert that 
   for all $T'<T$ and all $\tilde S\in \CRidev\np{\overline{\Omega}\times[0,T']}$
    \begin{equation}\label{eq:555}
    	\begin{aligned}
    		\int_0^{T'}\!\!\int_\Omega \bb{ {-}\Stens:\pt\tilde S
    			&+V\cdot\grad\Stens:\tilde S
    			+(\Stens\rottensor\np{V}-\rottensor\np{V}\Stens):\tilde S
    			\\
    			&+\partial\potential(\Stens):\tilde S +\gamma\grad\Stens:\grad\tilde S
    			-\eta\straintensor(\Vvel):\tilde S }\dx\dt
    		\\&\hspace{5em}=\int_\Omega\Stens_0:\tilde S(0)\dx
    		-\int_\Omega\Stens(T'):\tilde S(T')\dx.
    	\end{aligned}
    \end{equation}
The proof of this assertion follows from a  standard extension and approximation argument applied to the test functions $\tilde S$: given $\tilde S\in \CRidev\np{\overline{\Omega}\times[0,T']}$,
we can find an extension $\hat S\in \CRcidev(\OmegaTzerocl)$ of $\tilde S$. 
We then choose $\theta\in C^\infty(\mathbb{R})$, $\theta'\ge0$, $\theta=0$ on $(-\infty,-1]$, $\theta=1$ on $[0,\infty)$, let $\theta_{T',\delta}(t):=\theta(\frac{T'-t}{\delta})$ and define for $0<\delta\ll T-T'$ the function $\Psi_\delta(t,x):=\theta_{T',\delta}(t)\hat S(t,x)\in\CRcidev(\OmegaTzerocl)$, which satisfies $\Psi_\delta(t,\cdot)=\tilde S(t,\cdot)$ for all $t\in[0,T']$ and $\Psi_\delta(t,\cdot)=0$ for all $t\in[T'+\delta,T)$.
Inserting $\Psi_\delta$ as a test function in the weak formulation~\eqref{eq:weaksol_S}, 
performing standard manipulations and sending $\delta\to0$, we arrive at~\eqref{eq:555}. 
(For more details on the limit $\delta\to0$, we refer the proof of Prop.~\ref{prop:edin.genP} in Section~\ref{ssec:pf.eni}, where related arguments are carried out in a setting of lower regularity.)
  

Subtracting inequality~\eqref{est:EnergyInequality_S_combined} (at time $t=T'$)
we find, upon rearranging terms,
  \begin{equation}
   \nonumber 
    \begin{aligned}
      &\int_0^{T'}\!\!\int_\Omega [ -\Stens:\pt\tilde S
      +\partial\potential(\Stens):(\tilde S {-}S) +\gamma\grad\Stens:\grad(\tilde
      S {-} S) +V\cdot\grad\Stens:\tilde S
      \\ 
      &\qquad\qquad\qquad\qquad\quad
      +(\Stens\rottensor\np{V}-\rottensor\np{V}\Stens):\tilde S
      -\eta\straintensor(\Vvel):(\tilde S{-}S)\,]\dx\dt \\
     & \geq-\tfrac12\norm{\Stens_0}_2^2
      +\int_\Omega\Stens_0:\tilde S(0) \dx
      +\tfrac12\norm{S(T')}_2^2-\int_\Omega\Stens(T'):\tilde
      S(T')\dx.
    \end{aligned}
  \end{equation}
By the density of $\CRidev\np{\overline{\Omega}\times[0,T']}$ in $Z_{T'}$,  the last inequality continues to hold for all $\tilde S\in Z_{T'}$.
  The assertion is now obtained by adding
  $ \int_0^{T'}\!\int_\Omega \pt\tilde\Stens:\tilde S\,\dd x\dd t = \tfrac12
  \|\tilde S(T')\|_{2}^2 -\tfrac12 \|\tilde S(0)\|_{2}^2$,
  and using the fact that $  \tfrac{1}{2}\|\tilde S(T')-S(T')\|_{2}^2\ge0$ as well as the inequality $ \mathcal{P} (\tilde S(t)) - \mathcal{P} (S(t))\ge\int_\Om\partial \mathcal{P} (S(t)):(\tilde S(t){-}S(t))\,\dd x$.
\\[0.4em]
\underline{Part (B):} 
%
For a generalized solution $(V,S)$
with the smoothness as in \eqref{eq:Cond.Gen.Weak} we have
\begin{equation}
    \label{eq:beta.def}
\beta :=-\jaumannder{\Stens} 
	+\gamma\Delta\Stens +\eta \straintensor(\Vvel) \quad 
   \in \LR{2}([0,T'];\LRdev{2}(\Om)). 
\end{equation}
As a consequence of~\eqref{eq:varin}, we further note that $\mathscr{P}(S)=\int_0^{T'}\mathcal{P}(S)\,\dd t<\infty$.

Using the Zaremba--Jaumann identity \eqref{eq:Jaum.Identity} and the definition of
$\beta$, we find
\begin{align*}
&\int_0^{T'} \!\!\int_\Omega \Big( V\cdot \nabla S : \tilde S +
\big(SW(V)-W(V)S\big) : \tilde S \dx \dt 
\\
&\qquad
=\int_0^{T'} \!\!\int_\Omega \big(\jaumannder{S}{-}\partial_t S):\tilde S \dx
\dt = \int_0^{T'} \!\!\int_\Omega \big(\jaumannder{S}{-}\partial_t
S):(\tilde S{-}S) \dx \dt  
\\
&\qquad=\int_0^{T'} \!\!\int_\Omega \big({-}\partial_t S - \beta +\gamma \Delta S +
\eta \straintensor(\Vvel) \big): (\tilde S{-}S) \dx \dt . 
\end{align*} 
Inserting this identity into \eqref{eq:varin}, we are left with the variational
inequality
\begin{align}\label{eq:106}
	\begin{aligned}
		\int_0^{T'} \!\!\int_\Omega \Big( (\partial_t \tilde S{-}\partial_t S) : (\tilde
		S{-}S) - \beta : (\tilde S{-}S) \Big) \dx \dt 
		&+\mathscr{P}(\tilde S)-\mathscr{P}(S)
		\\
		&
		\geq      
		-\tfrac{1}{2}\|\tilde S( 0)-\Stens_0\|_{2}^2
	\end{aligned}
\end{align}
 for all $\tilde S\in \TestFn$,
where we recall the definition of $\mathscr P$ in~\eqref{eq:def.scrP}.
Given $\tilde R\in \TestFn$ with $\tilde R(0)=0=\tilde R(T')$, we choose in~\eqref{eq:106}
the test function $\tilde S=S+\tilde R$, which by~\eqref{eq:Cond.Gen.Weak} lies in $\TestFn$ and moreover satisfies $\tilde S(0)=S_0$,  to infer
\begin{align}
	\label{eq:scrP}
	\mathscr{P}(S{+}\tilde R)-\mathscr{P}(S)\geq \int_0^{T'}\!\!\int_\Om 
	\beta:\tilde R \dx \dt
\end{align}
for all such $\tilde R$. 

We assert that by means of an approximation argument, ineq.~\eqref{eq:scrP} can be extended to general $\tilde R\in \TestFn$, not necessarily vanishing at the boundary of $(0,T')$. To see this, we pick a sequence $\{\theta_j\}\subset C^\infty_0((0,T'))$ with $0\le \theta_{j}\le\theta_{j+1}\le 1$ for all $j\in\mathbb{N}$ and such that  $\lim_{j\to\infty}\theta_{j}(t)=1$ for all $t\in(0,T')$. We then infer from ineq.~\eqref{eq:scrP} for general $\tilde R\in \TestFn$
\begin{align}
	\label{eq:scrPtheta}
	\int_0^{T'}\big(\mathcal{P}(S{+}\theta_j\tilde R)-\mathcal{P}(S)\big)\,\dd t\geq \int_0^{T'}\!\!\theta_j(t)\!\int_\Om 
	\beta:\tilde R \dx \dt. 
\end{align}
Since $\mathcal{P}$ is convex, we have for every $\theta=\theta_j(t)\in[0,1]$
\[
\mathcal{P}(S{+}\theta\tilde R)-\mathcal{P}(S)=\mathcal{P}\big(\theta(S{+}\tilde R)+(1-\theta)S\big)-\mathcal{P}(S)
\le \theta \mathcal{P}(S{+}\tilde R)-\theta \mathcal{P}(S).
\]
Inserting this inequality into~\eqref{eq:scrPtheta} gives
\begin{align*}
	\int_0^{T'}\theta_j(t) \mathcal{P}(S{+}\tilde R)\,\dd t-\int_0^{T'}\theta_j(t)\mathcal{P}(S)
	\,\dd t
	\geq \int_0^{T'}\!\!\theta_j(t)\!\int_\Om 
	\beta:\tilde R \dx \dt. 
\end{align*}
Invoking the monotone convergence theorem for the first term in the last line and using dominated convergence for the remaining two time integrals, we can take the limit $j\to\infty$ in the last inequality and arrive at~\eqref{eq:scrP}
for general $\tilde R\in \TestFn$.

Due to the approximation property \eqref{eq:calP.appr.prop} of $\mathscr
P$, we can further extend~\eqref{eq:scrP} to general 
$R \in  \LR{2}(0,T';\LR{2}_\delta(\Omega))$. 
Indeed, letting $\tilde S:=S+R \in  \LR{2}(0,T';\LR{2}_\delta(\Omega))$, property~\eqref{eq:calP.appr.prop} provides us with a sequence $(\tilde S_n)\subset\TestFn$ such that $ \tilde S_n \rightharpoonup
S+R$ in $\LR{2}(0,T';\LR{2}_\delta(\Omega))$ and  $\mathscr P(\tilde S_n)\to \mathscr P(S+R)$. Hence, inserting
$\tilde R=\tilde R_n:=\tilde S_n-S$ in \eqref{eq:scrP} and passing
to the limit $n\to \infty$ we obtain
\[
 \mathscr{P}(S{+}R) \geq \mathscr{P}(S) + \int_0^{T'}\!\!\int_\Om 
  \beta: R \dx \dt \quad \text{ for all }
  R \in \LR{2}(0,T';\LR{2}_\delta(\Omega)) .
\]
But this is exactly the definition of $\beta \in \partial \mathscr P(S)$, and
the special definition of $\mathscr P$ in terms of $\mathcal P$ (cf.~\eqref{eq:def.scrP}) implies
$\beta(t)\in \partial\mathcal P(S(t))$ a.e.\ on $(0,T')$.

The definition of $\beta$ in \eqref{eq:beta.def} implies the desired weak equation \eqref{eq:weak.nonsmo}.
\end{proof}

\subsection{Properties of the Moreau envelope}

Recall the definition of the Moreau envelope $\{\mathcal{P}_\ve\}$ in~\eqref{eq:YosMorEnvel}.
As an immediate consequence of~\eqref{eq:YosMorEnvel}, we find that  
$\mathcal{P}_\ve(S)\le \mathcal{P}(S)$ and hence  
\begin{align}\label{eq:101}
	\limsup_{\ve\to0}\int_0^T\mathcal{P}_\ve(S(t))\,\dd t\le
	\int_0^T\mathcal{P}(S(t))\,\dd t\quad\text{ for all }S\in
	\LR{2}(0,T;\LRdev{2}(\Omega)). 
\end{align}
The proof of Theorem~\ref{thm:genP} further makes use of the following version of the classical approximation properties of the Moreau envelope \cite{BC_2017,Roubicek_2013}. 
\begin{lem}\label{l:wlsc}
	Let $\potential$ be as in~\eqref{eq:propP}.  The Moreau
	envelope $\{\mathcal{P}_\ve\}$ of $\mathcal{P}$, as defined in~\eqref{eq:YosMorEnvel}, satisfies the inequality
	\begin{align}
		\label{eq:wlsc}
		\liminf_{\ve\to0} \int_0^T\mathcal{P}_\ve(S_\ve(t))\,\dd t
		\ge 	\int_0^T\mathcal{P}(S(t))\,\dd t 
	\end{align}
	whenever $S_\ve\rightharpoonup S$ in $\LR{2}(0,T;\LRdev{2}(\Omega))$.
\end{lem}

\begin{proof}
	Let $\delta>\ve>0$. Then, by definition,
	$\mathcal{P}_\ve\ge\mathcal{P}_\delta$.  Hence
	\begin{align}
		\label{eq:103}
		\liminf_{\ve\to0} \int_0^T\!\mathcal{P}_\ve(S_\ve(t))\,\dd t
		\ge 	 \liminf_{\ve\to0} \int_0^T\!\mathcal{P}_\delta(S_\ve(t))\,\dd t
		\ge \int_0^T\!\mathcal{P}_\delta(S(t))\,\dd t.
	\end{align}
	The second step follows from the fact that for $\delta>0$ the functional
	\begin{align*}
		F_\delta:	\LR{2}(0,T;\hs)\ni S\mapsto \int_0^T\mathcal{P}_\delta(S(t))\,\dd t
	\end{align*}
	is convex and continuous, and thus weakly lower semicontinuous.  The
	convexity of $F_\delta$ is inherited from the convexity of
	$\mathcal{P}_\delta$, while continuity follows from standard theory on
	Nemytskii operators (see e.g.~\cite[Theorem~1.43]{Roubicek_2013}) together
	with the growth condition
	$ 0\le \mathcal{P}_\delta(S)\le \|S\|_{\hs}^2/(2\delta)$, which is a
	consequence of the definition of the Moreau envelope.
	
	To show the assertion, it now remains to prove that
	\begin{align}
		\label{eq:105}
		\lim_{\delta\to0}\int_0^T\!\mathcal{P}_\delta(S(t))\,\dd t 
		= \int_0^T\!\mathcal{P}(S(t))\,\dd t.
	\end{align}
	By~\cite[Proposition~12.33 (ii)]{BC_2017},
	$\mathcal{P}_\delta(S(t))\to \mathcal{P}(S(t))$ for a.e.~$t\in(0,T)$. The
	nonnegativity of $\mathcal{P}_\delta$ and Beppo Levi's monotone
	convergence imply the identity~\eqref{eq:105}. Together with \eqref{eq:103}
	the desired assertion \eqref{eq:wlsc} follows. 
\end{proof}

\subsection{Proof of the main result (Theorem~\ref{thm:genP})}\label{ssec:proof.main}

In this subsection, we will prove our main theorem. 
Let us start by highlighting the following elementary result, which presents the crucial idea for passing to
the limit, along approximate solutions, in the tensorial evolutionary variational equality and in particular in the
nonlinear terms arising from the Zaremba--Jaumann
derivative. 

\begin{lem} \label{le:SW(V)}
Let $V_\ve=(V^\ve_i)$ and $S_\ve=(S^\ve_{jk})$ 
satisfy the conditions 
\[
V_\ve \to  V \text{ in } \LR 2(\Om\tim (0,T')), 
\qquad 
V_\ve \rightharpoonup V \text{ and }  S_\ve \rightharpoonup S \text{ in } \LR
2(0,T';\HSR{1}(\Omega)),
\]
let $V_\ve |_{\partial \Omega} = g \in \LR{2}(0,T';\LR{2}(\partial
\Omega))$ be fixed,
and assume that $\|S_\ve\|_{\LR{\frac{10}{3}}(\Om\tim (0,T'))}\le C$.

Then, for all $i,j,k,l\in \{1,2,3\}$  we have 
\begin{equation}
  \label{eq:Jaum.viaIBP}
  \lim_{\ve \to 0} \int_0^{T'} \!\! \int_\Om S^\ve_{ij} \partial_k V^\ve_l \,\psi \dx \dt
  = \int_0^{T'} \!\! \int_\Om S_{ij} \partial_k V_l \,\psi \dx \dt 
\end{equation}
for all $\psi \in \LR 5(\Om\tim (0,T'))$. 
\end{lem}

\begin{proof}
	We first show that~\eqref{eq:Jaum.viaIBP} holds for $\psi\in C^1(\overline{\Om}\times[0,T'])$.
	In this case, integration by parts with respect to the spatial variable gives
\[
\int_0^{T'} \!\! \int_\Om S^\ve_{ij} \partial_k V^\ve_l \,\psi \dx\dt = 
\int_0^{T'} \!\! \int_{\partial\Om} S^\ve_{ij}  g_l \nvec_k \,\psi \,\dd \sigma\dt -
\int_0^{T'} \!\! \int_\Om \partial_k \big( S^\ve_{ij}\,\psi\big)  V^\ve_l  \dx \dt, 
\]
where we have already exploited the boundary conditions $V_\ve = g$ on
$\partial\Om$.

Because of the continuity of the trace operator from $\HSR{1}(\Om)$ to
$\LR{2}(\partial\Om)$, we have $S_\ve\rightharpoonup S$ in
$\LR{2}(0,T';\LR{2}(\partial\Om))$ and can pass to the limit in the first term
on the right-hand side. 
For the last term we use the weak convergence of $S_\ve $ to $S$ in
$\LR{2}(0,T';\HSR{1}(\Omega))$ as well as the strong convergence of $V_\ve$ to $V$ in $\LR 2(\Om\tim (0,T'))$. 
Undoing the spatial integration by parts, we obtain the desired result for $\psi\in \CR{1}(\overline{\Om}\times[0,T'])$.

The validity of~\eqref{eq:Jaum.viaIBP} for general $\psi\in \LR 5(\Om\tim (0,T'))$
is now a consequence of the density of $\CR{1}(\overline{\Om}\times[0,T'])$ in $\LR 5(\Om\tim (0,T'))$ and the fact that, by H\"older's inequality (with $\frac{3}{10}+\frac{1}{2}=\frac{4}{5}$), 
the sequence $ \{S^\ve_{ij} \partial_k V^\ve_l\}_\ve$ is $\ve$-uniformly bounded in $\LR{ \frac{5}{4}}(\Om\tim (0,T'))$.
\end{proof}

We are now in a position to complete the proof of Theorem \ref{thm:genP} and show
existence of generalized solutions to \eqref{sys:PDE}.

\begin{proof}[Proof of Theorem~\ref{thm:genP}]\label{p:proofThmP}
  For $\ve\in(0,1]$ consider the Moreau envelope $\mathcal{P}_\ve$ for 
  $\mathcal{P}$ as introduced in Section \ref{sec:Preliminaries}, and denote by
  $(V_\ve,S_\ve)$ the weak solution constructed in
  Theorem~\ref{thm:existence_weaksol} with $\mathcal{P}$ replaced by
  $\mathcal{P}_\ve$.  By estimate~\eqref{est:weaksol_bounds}, we have the
  $\ve$-uniform bound
  \begin{equation}\label{est:eps.uniform}
    \begin{aligned}
      &\norm{\Vvel_\ve}_{\LR{\infty}\np{0,T';\LR{2}(\Omega)}}
      +\norm{\Vvel_\ve}_{\LR{2}\np{0,T';\HSR{1}(\Omega)}}
      +\norm{\pt\Vvel_\ve}_{\LR{1}\np{0,T';\np{\HSRNsigma{1}(\Omega)}'}}
      \\
      &\qquad+\norm{\Stens_\ve}_{\LR{\infty}\np{0,T';\LR{2}(\Omega)}}
      +\norm{\Stens_\ve}_{\LR{2}\np{0,T';\HSR{1}(\Omega)}}
      +\int_0^{T'}\!\!\int_\Omega\potential_\ve(\Stens_\ve)\dx\dt \leq M_{T'}
    \end{aligned}
  \end{equation}
  for all $T'<T$.  Hence, there exists a sequence $\ve\to0$ (not relabeled)
  and a pair $(V,S)$ such that for all $T'<T$ one has $(V,S)\in
  \LHprime\times\XTprime$ and
\[
\begin{aligned}
	\Vvel_\ve&\rightharpoonup\Vvel && \tin \LR{2}(0,T';\HSR{1}(\Omega))^3,
	\\
	\Stens_\ve&\rightharpoonup\Stens && \tin \LR{2}(0,T';\HSR{1}(\Omega))^{3\times3},
	\\
	\Vvel_\ve&\overset{\ast}{\rightharpoonup}\Vvel && \tin \LR{\infty}(0,T';\LRsigma{2}(\Omega)),
	\\
	\Stens_\ve&\overset{\ast}{\rightharpoonup}\Stens && \tin \LR{\infty}(0,T';\LRdev{2}(\Omega)),
	\\	\Vvel_\ve&\to\Vvel && \tin \LR{2}(0,T';\LR{2}(\Omega))^3,
\end{aligned}
\]
where, as in the proof of Theorem \ref{thm:existence_PDEmodified}, the strong
convergence of $(\Vvel_\ve)$ is obtained from an Aubin--Lions compactness
result.

The passage to the limit $\ve\to0$ in the weak form~\eqref{eq:weaksol_V} of the
equation for the velocity field $V_\ve$ follows from standard arguments based
on the above convergence properties. As a result, the limiting vector field $V$
satisfies eq.~\eqref{eq:weaksol_V}. Moreover, the fact that
$\restriction{\Vvel_\ve}{\partial\OmegaT}=g$ combined with the above
convergence properties easily yields $\restriction{\Vvel}{\partial\OmegaT}=g$.
Thus, it remains to show that $S$ satisfies inequality~\eqref{eq:varin} for all
$\tilde S\in\TestFn$.
	
By Lemma~\ref{l:vi.aux} (A), $S_\ve$ satisfies the variational inequality
\begin{equation}\label{eq:207}
  \begin{multlined}
    \int_0^{T'}\!\int_\Omega \pt\tilde\Stens:(\tilde S-\Stens_\ve)
    +\gamma\grad\Stens_\ve:\grad(\tilde S-\Stens_\ve)\,\dd x\dd t +
    \int_0^{T'}\potential_\ve(\tilde\Stens)-\potential_\ve(\Stens_\ve)\,\dd t
    \\
    +\int_0^{T'}\!\int_\Omega V_\ve\cdot\grad\Stens_\ve:\tilde S
    +(\Stens_\ve\rottensor\np{V_\ve} -\rottensor\np{V_\ve}\Stens_\ve):\tilde S
    -\eta\straintensor(\Vvel_\ve):(\tilde S-\Stens_\ve) \dx\dd t \\ 
\hspace{4em}\geq
 -\tfrac{1}{2}\|\tilde S( 0)-\Stens_0\|_{2}^2.
  \end{multlined}
\end{equation}
We will deduce ineq.~\eqref{eq:varin} by estimating the $\limsup_{\ve\to0}$ of
the left-hand side.
	
First, the weak convergence $S_\ve\rightharpoonup S$ in
$\LR{2}(0,{T'};\HSR{1}(\Omega))^{3\times3}$ implies that
\begin{multline*}
  \int_0^{T'}\!\int_\Omega \pt\tilde\Stens:(\tilde S-\Stens)
  +\gamma\grad\Stens:\grad(\tilde S-\Stens)\,\dd x\dt \\\ge \limsup_{\ve\to0}
  \int_0^{T'}\!\int_\Omega \pt\tilde\Stens:(\tilde S-\Stens_\ve)
  +\gamma\grad\Stens_\ve:\grad(\tilde S-\Stens_\ve)\dx \dt,
\end{multline*}
where we used weak upper semicontinuity of the concave quadratic term.
	
For the second term on the left-hand side of~\eqref{eq:207}, we use the bound
\begin{align*}
  \limsup_{\ve\to0}
  \int_0^{T'}\!\!\big(\potential_\ve(\tilde\Stens)-\potential_\ve(\Stens_\ve)\big)\dt 
  \le	\int_0^{T'}\!\!\big(\potential(\tilde\Stens)-\potential(\Stens)\big)\dt
\end{align*}
which is consequence of Lemma~\ref{l:wlsc} and inequality \eqref{eq:101}.

Further note that $V_\ve \to V$ in $\LR{2}(0,T';\LR{2}(\Omega))$ and 
$S_\ve \rightharpoonup S$ in $ \LR{2}(0,T';\HSR{1}(\Omega)) $ imply 
\begin{align*}
  \lim_{\ve\to0}\int_0^{T'}\!\int_\Omega V_\ve\cdot\grad\Stens_\ve:\tilde S \dx\dd t
  =\int_0^{T'}\!\int_\Omega V\cdot\grad\Stens:\tilde S \dx\dt.
\end{align*}

The term $(\Stens_\ve\rottensor\np{V_\ve}
-\rottensor\np{V_\ve}\Stens_\ve):\tilde S + \eta \straintensor(\Vvel_\ve) :S_\ve  $
consists of a finite linear 
combination of terms handled in Lemma \ref{le:SW(V)}. 
This lemma can be applied thanks to the convergence properties of $\np{V_\ve,S_\ve}$ and the interpolation~\eqref{eq:interpol} ensuring the boundedness of $(S_\ve)$ in $\LR{\frac{10}{3}}(\Om\tim (0,T'))$. 
 Hence we can pass to the
limit with all remaining parts in the left-hand side.

The above observations allow us to estimate the $\limsup_{\ve\to0}$ of the
left-hand side of~\eqref{eq:207} above by the left-hand side of
inequality~\eqref{eq:varin}.  
Hence $(\Vvel,\Stens)$ is a generalized solution to
\eqref{sys:PDE}--\eqref{eq:InitialConditions} in the sense of Definition
\ref{def:varsol}.

Since \eqref{eq:GenSol.EDI} directly follows from \eqref{est:EnergyInequality_v_combined}
and \eqref{eq:edin.S.genP}, 
and the latter follows from Proposition \ref{prop:edin.genP},
which is proved below,
it remains to establish the partial energy-dissipation inequality
\eqref{est:EnergyInequality_v_combined}. 
But this is a simple consequence of the previously
established partial energy-dissipation inequality. In particular, we note that the
boundary extension $w$ constructed for Theorem~\ref{thm:existence_weaksol}
depends only on $g$ and, hence, is independent of the regularization parameter
$\ve$. Thus, we can use the partial energy-dissipation inequality
\eqref{est:EnergyInequality_v_combined} for $v_\ve=V_\ve-w$ and $S_\ve$. With the given
weak and strong convergences, we can pass to the limit $\ve \to 0$ and obtain
the corresponding inequality for the limits $v=V{-}w$ and $S$. 
\end{proof}

Finally, we show that, under suitable decay assumptions on the data,
the total energy of the solution remains bounded.

\begin{proof}[Proof of Corollary \ref{cor:GlobalSolutionsEnergy}]
The pair $\np{\vvel,\Stens}$ satisfies the energy inequality \eqref{eq:GenSol.EDI},
and $\wvel$ is constructed in such a way that Assumption \ref{asm:extension} 
\ref{item:extension_iii} is satisfied.
Proceeding as in the proof of Theorem \ref{thm:existence_PDEmodified},
we can then estimate the right-hand side of \eqref{eq:GenSol.EDI}
and use an absorption argument to conclude
\begin{equation}\label{est:energy.bounded.proof}
\begin{aligned}
&\norm{\vvel}_{\LR{\infty}(0,T';\LR{2}(\Omega))}^2
+\norm{\grad\vvel}_{\LR{2}\np{\OmegaTprime}}^2
+\norm{\Stens}_{\LR{\infty}(0,T';\LR{2}(\Omega))}^2
\\
&\qquad\qquad\qquad\qquad\qquad\qquad
+\norm{\grad\Stens}_{\LR{2}\np{\OmegaTprime}}^2
+\int_0^{T'}\potential(\Stens)\dt
\\
&\quad
\leq C\bp{
\norm{\Vvel_0-\wvel(0)}_2^2+\norm{\Stens_0}_2^2 
+\norm{F_0}_{\LR{1}(0,T';\LR{2}(\Omega))}^2
+\norm{F_1}_{\LR{2}(\OmegaTprime)}^2 
\\
&\qquad\qquad\quad
+\norm{\widetilde{F}_1}_{\LR{2}(\OmegaTprime)}^2 
+\norm{\wvel}_{\LR{4}(\OmegaTprime)}^4
+\norm{\grad\wvel}_{\LR{1}(0,T';\LR{2}(\Omega))}^2}
\end{aligned}
\end{equation}
for any $T'\in(0,\infty)$.
To further estimate the right-hand side, recall from
the proof of Theorem \ref{thm:existence_weaksol} in Subsection \ref{subsec:existence.weaksol}
that $\wvel\in\CR{0}(0,T';\LR{2}(\Omega))$
and $\Div\widetilde{F}_1=\pt\wvel-\mu\Delta\wvel\in\LR{2}(0,T';\HSR{-1}(\Omega))$.
We further obtain
\[
\norm{\wvel(0)}_{2}
\leq\norm{\wvel}_{\LR{\infty}(0,T';\LR{2}(\Omega))}
\leq C\norm{g}_{\LR{\infty}(0,T';\HSR{1/2}(\partial\Omega))}
\]
and
\[
\begin{aligned}
\norm{\widetilde{F}_1}_{\LR{2}(\OmegaTprime)}
&\leq C\bp{\norm{\pt\wvel}_{\LR{2}(0,T';\HSR{-1}(\Omega))}
+\norm{\wvel}_{\LR{2}(0,T';\HSR{1}(\Omega))}}
\\
&\leq C\bp{\norm{\pt g}_{\LR{2}(0,T';\HSR{-1/2}(\partial\Omega))}
+\norm{g}_{\LR{2}(0,T';\HSR{1/2}(\partial\Omega))}},
\end{aligned}
\]
where we used \eqref{est:w.lin} in the respective last estimate.
By means of \eqref{est:w.lin} 
we can also estimate the last two terms in \eqref{est:energy.bounded.proof}
in terms of $g$.
Then a standard interpolation argument shows that 
the right-hand side of \eqref{est:energy.bounded.proof} 
is bounded by a constant independent of $T'\in(0,\infty)$.
This shows $\np{\vvel,\Stens}\in\LH\times\XT$ for $T=\infty$.
Since we have $\Vvel=\vvel+\wvel$, and $\wvel\in\LH$ 
for $T=\infty$
by \eqref{est:w.lin}, 
this completes the proof.
\end{proof}

\subsection{Partial energy inequality}\label{ssec:pf.eni}

\begin{proof}[Proof of Proposition \ref{prop:edin.genP}]
	Observe that choosing $\tilde S\equiv0$ in~\eqref{eq:varin} shows that 
	$\int_0^{T'}\!\!\potential(\Stens)\dtau<\infty$. Let us further note that since $S\in \LR{\infty}(0,T';\LR{2}(\Om))$, almost every $T'\in(0,T)$ is a left Lebesgue point of $t\mapsto S(t)\in \LR{2}(\Om)$. 
	
	Extend now $S$ by zero for $t<0$ and consider for  $\kappa>0$ 
	\begin{align*}
		S_\kappa(t)=\kappa^{-1}\int_{t-\kappa}^tS(\tau)\,\dd\tau.
	\end{align*}
	Further let $\eta\in \CRi(\mathbb{R};[0,1])$ be nondecreasing, $\eta(t)=0$ for $t\le-1$ and $\eta(t)=1$ for $t\ge0$, and define $\eta_\delta(t):=\eta_{\delta}^{(T')}(t):=\eta((t-T')/\delta)$.
	We then choose in~\eqref{eq:varin} the  test function $\tilde S:=\tilde S_{\kappa,\delta}:=\eta_\delta S_\kappa\in Z_{T'}$, where $\delta\in(0,\delta_*]$ and $\kappa\in(0,\kappa_*]$ are chosen sufficiently small and, in particular,
	such that $\tilde S_{\kappa,\delta}(0)=0$.
	
	Since $\mathcal{P}$ is convex  with $\mathcal{P}(0)=0$ and $0\le \eta_\delta\le 1$, we can estimate
	using Jensen's inequality
	\begin{align*}
		\int_0^{T'}\!\!\potential(\tilde \Stens_{\kappa,\delta})\dtau \le
		\kappa^{-1}\int_{T'-\delta}^{T'}\!\!\eta_\delta(t) \int_{t-\kappa}^t\mathcal{P}(S(\tau))\dtau\,\dd t\le \kappa^{-1}\int_{0}^{T'}\!\!\potential(\Stens)\dtau\,\cdot\,\delta,
	\end{align*}
	where we also used the nonnegativity of $\mathcal{P}$.
	Since  $\int_0^{T'}\!\!\potential(\Stens)\dtau<\infty$, the last line implies that $\lim_{\delta\to0}\int_0^{T'}\!\potential(\tilde \Stens_{\kappa,\delta})\dtau=0$ for any $\kappa\in(0,\kappa_*]$.
	
	Let us next turn to the integral involving the time derivative. Using the fact that $\eta_\delta(T')=1$, we find
	\begin{align*}
		\int_0^{T'}\!\int_\Omega \pt\tilde\Stens:(\tilde S-\Stens)\,\dd x\,\dd t
		=  \frac{1}{2}\|S_\kappa(T')\|_2^2
		&-\int_0^{T'}\eta_\delta'(t)\int_\Omega S_\kappa:S\,\dd x\,\dd t 
		\\&
		-\int_{T'-\delta}^{T'} \!\eta_\delta(t) \!\int_\Omega\pt S_\kappa:S\,\dd x\,\dd t.
	\end{align*}
	Since $S\in \LR{\infty}(0,T';\LR{2}(\Om))$, we easily see that the term in the last line vanishes as $\delta\to0$.
	Furthermore, we have the following convergence results,  valid for almost all $T'\in(0,T):$
	\begin{align*}
		&\lim_{\delta\to0}\int_0^{T'}\eta_\delta'(t)\int_\Omega S_\kappa:S\,\dd x\,\dd t 
		=\int_\Omega S_\kappa(T'):S(T')\,\dd x,
		\\&\lim_{\kappa\to0}\int_\Omega S_\kappa(T'):S(T')\,\dd x = \|S(T')\|_2^2,
		\\&\lim_{\kappa\to0}\,\frac{1}{2}\|S_\kappa(T')\|_2^2=\frac{1}{2} \|S(T')\|_2^2.
	\end{align*}
	Thus, for almost all $T'$ we obtain
	\begin{align*}
		\lim_{\kappa\to0}\;\lim_{\delta\to0}\;
		\int_0^{T'}\!\!\int_\Omega \pt\tilde\Stens_{\kappa,\delta}:(\tilde S_{\kappa,\delta}-\Stens)\,\dd x\,\dd t = - \frac{1}{2} \|S(T')\|_2^2.
	\end{align*}
	All remaining integrals in~\eqref{eq:varin} involving $\tilde S$ converge to zero as $\delta\to0$, as long as $\kappa$ is positive. 
	Thus, sending first $\delta\to0$ in~\eqref{eq:varin} (with $\tilde S=\tilde S_{\kappa,\delta}$), and taking subsequently the limit $\kappa\to0$, we arrive at~\eqref{eq:edin.S.genP}.
\end{proof}

\section*{Acknowledgements}

The authors would like to thank the anonymous reviewers for their careful reading of the article as well as
their valuable and precise comments, which significantly improved the presentation of this manuscript.


\newcommand{\etalchar}[1]{$^{#1}$}

\end{document}